\documentclass[a4paper,11pt]{article}

\usepackage{ucs}
\usepackage[latin9]{inputenc}
\usepackage{graphicx}
\usepackage{amsfonts}
\usepackage{dsfont}
\usepackage{amssymb}
\usepackage{amsmath}
\usepackage{amsthm}
\usepackage{enumerate}
\usepackage{stmaryrd}
\usepackage{fullpage}
\usepackage{ifthen}
\usepackage{subfigure}
\usepackage{epic}
\usepackage{authblk}
\usepackage{textcomp}
\usepackage{mathrsfs}
\usepackage[hypertexnames=false,colorlinks=true,linkcolor=blue,citecolor=blue]{hyperref}
\usepackage[numbers,comma,square,sort&compress]{natbib}

\usepackage{color}
\usepackage{titlesec}

\graphicspath{{eps/}{pdf/}}

\newcommand{\bqq}{\begin{equation}}
\newcommand{\eqq}{\end{equation}}
\newcommand{\bqs}{\begin{equation*}}
\newcommand{\eqs}{\end{equation*}}

\newcommand{\R}{\mathbb{R}}

\newcommand{\md}{\mathrm{d}}

\newcommand{\K}{\mathcal{K}}

\newtheorem{lem}{Lemma}[section]

\newtheorem{prop}[lem]{Proposition}
\newtheorem{cor}[lem]{Corollary}
\newtheorem{rmk}[lem]{Remark}
\newtheorem{defi}[lem]{Definition}
\newtheorem{hyp}[lem]{Hypothesis}
\newtheorem{conj}[lem]{Conjecture}

 {\begin{trivlist}\item[]\textbf{Proof#1 }}%
 {\hspace*{\fill}$\rule{0.3\baselineskip}{0.35\baselineskip}$\end{trivlist}}

  {\begin{trivlist}\item[]{\bf Hypothesis #1 }\em}{\end{trivlist}}

\numberwithin{equation}{section}

\title{Asymptotic behavior of nonlocal bistable reaction-diffusion equations}

\author[1]{Christophe Besse}
\author[2]{Alexandre Capel}
\author[1]{Gr\'egory Faye\footnote{ \texttt{gregory.faye@math.univ-toulouse.fr}}}
\author[3]{Guilhem Fouilh\'e}
\affil[1]{Institut de Math\'ematiques de Toulouse ; UMR5219, Universit\'e de Toulouse ; CNRS UPS IMT, F-31062 Toulouse Cedex 9 France}
\affil[2]{Universit\'e de Montpellier, France}
\affil[3]{Universit\'e Paul Sabatier, Toulouse 3, France}

\begin{document}
\maketitle

\begin{center}
{\it Dedicated to the memory of Masayasu Mimura, in deep gratitude for his inspiration.}
\end{center}

\begin{abstract}
 In this paper, we study the asymptotic behavior of the solutions of nonlocal bistable reaction-diffusion equations starting from compactly supported initial conditions. Depending on the relationship between the nonlinearity, the interaction kernel and the diffusion coefficient, we show that the solutions can either: propagate, go extinct or remain pinned. We especially focus on the latter regime where solutions are pinned by thoroughly studying discontinuous ground state solutions of the problem for a specific interaction kernel serving as a case study. We also present a detailed numerical analysis of the problem.
\end{abstract}



\section{Introduction}

We investigate the long time behavior of the solutions of the following Cauchy problem
\bqq
\label{edp}
\left\{
\begin{split}
\partial_t u &= d\left(-u+\K *u\right)+f(u), \quad t>0 \quad x\in\R, \\
u(t=0,x) & = \mathds{1}_{[-\ell,\ell]}(x), \quad x\in\R,
\end{split}
\right.
\eqq
where $\K*u$ stands for the convolution on the real line with $d>0$ and $\ell>0$ some parameters. Throughout, the nonlinearity $f$ will be assumed to be of bistable type in the sense that it satisfies the following hypothesis.
\begin{hyp}\label{hypf}
The function $f:\R\to \R$ is smooth and there exists some $a\in(0,1)$ such that $f(0)=f(a)=f(1)=0$ with $f<0$ on $(0,a)$ and $f>0$ on $(a,1)$. We further assume that $f'(0)<0$, $f'(1)<0$ and $f'(a)>0$ together with the condition that $\int_{0}^1f(u)\md u>0$.
\end{hyp}

The above equation \eqref{edp} belongs to the class of reaction-diffusion equations which are widely used to describe the spatio-temporal evolution of species under diffusion and demographic effects. Here, the bistable nature of $f$ traduces a strong Allee effect. The diffusion is nonlocal and models long-range dispersal properties. The kernel $\K$ can be interpreted as the probability density of individuals moving from position $y$ to position $x$. Here, we have assumed that the environment is homogeneous which explains the convolution term and ensures the translation invariance of the equation. A this stage, we will make further natural assumptions on the kernel $\K$.

\begin{hyp}\label{hypK}
We require that the convolution kernel $\K$ satisfies
\begin{itemize}
\item[(i)] $\K,\K'\in L^1(\R)$, with $\int_\R \K(x)\md x =1$;
\item[(ii)] $0\leq \K(x)=\K(-x)<+\infty$ for all $x\in\R$;
\item[(iii)] $\int_\R\K(x)|x|\md x<+\infty$.
\end{itemize}
\end{hyp}

Under the above two hypotheses on $f$ and $\K$, the Cauchy problem \eqref{edp} is well-posed in the sense that there exists a unique solution $u\in\mathscr{C}^1\left((0,+\infty),L^1(\R)\cap L^\infty(\R) \right)\cap \, \mathscr{C}^0\left([0,+\infty),L^1(\R)\cap L^\infty(\R) \right)$ which is global and satisfies $0 < u(t,x) < 1$ for all $t>0$ and $x\in\R$ as a consequence of the comparison principle which applies in this case. Our aim is thus to characterize the asymptotic behavior of the solution $u$ as a function of the diffusion coefficient $d$, the nonlinearity $f$, the kernel $\K$ and the parameter $\ell$ which represents the initial implantation of the population.  In the purely local case, that is when the nonlocal diffusion is replaced by the Laplacian, it is well-known that a sharp threshold occurs. More precisely, at given fixed $d$ and bistable nonlinearity $f$, there exists a unique critical value $\ell_*>0$ such that the following dichotomy holds:
\begin{itemize}
\item for all $\ell\in(0,\ell_*)$, the solution $u$ goes extinct;
\item for all $\ell>\ell_*$, the solution $u$ propagates across the full domain.
\end{itemize}
The interpretation of the above result is somehow intuitive. If initially, the mass of the population is too small, due to the strong Allee effect, the population will go extinct. On the other hand, if the initial mass is large enough, the population will survive and spread across the full domain. These two results go back to the pioneer works of Kanel \cite{Kanel} and Flores \cite{Flores} (for {\bf extinction}) and Fife \& McLeod \cite{FMcL} (for {\bf propagation}). The fact that there is a sharp threshold is a more striking result and was only obtained recently by Zlat\v{o}s \cite{Zlatos}, and later refined by Du \& Matano  \cite{DM10}, Pol\'{a}\v{c}ik \cite{Pol11} and Muratov \& Zhong \cite{MZ13,MZ17}. At the threshold, the solution converges towards a non trivial steady state of the equation, often called a {\em ground state}. This scenario, which is unstable from a dynamical systems point of view, reflects a sort of {\bf stagnation} mechanism where the solution neither propagates nor goes extinct.  Obtaining quantitative estimates on the threshold value $\ell_*$ is a more delicate task and we refer to the recent work of Alfaro et al. \cite{ADF20} for some progress and references therein.

\begin{figure}[t!]
  \centering
 \subfigure[Extinction.]{ \includegraphics[width=.3\textwidth]{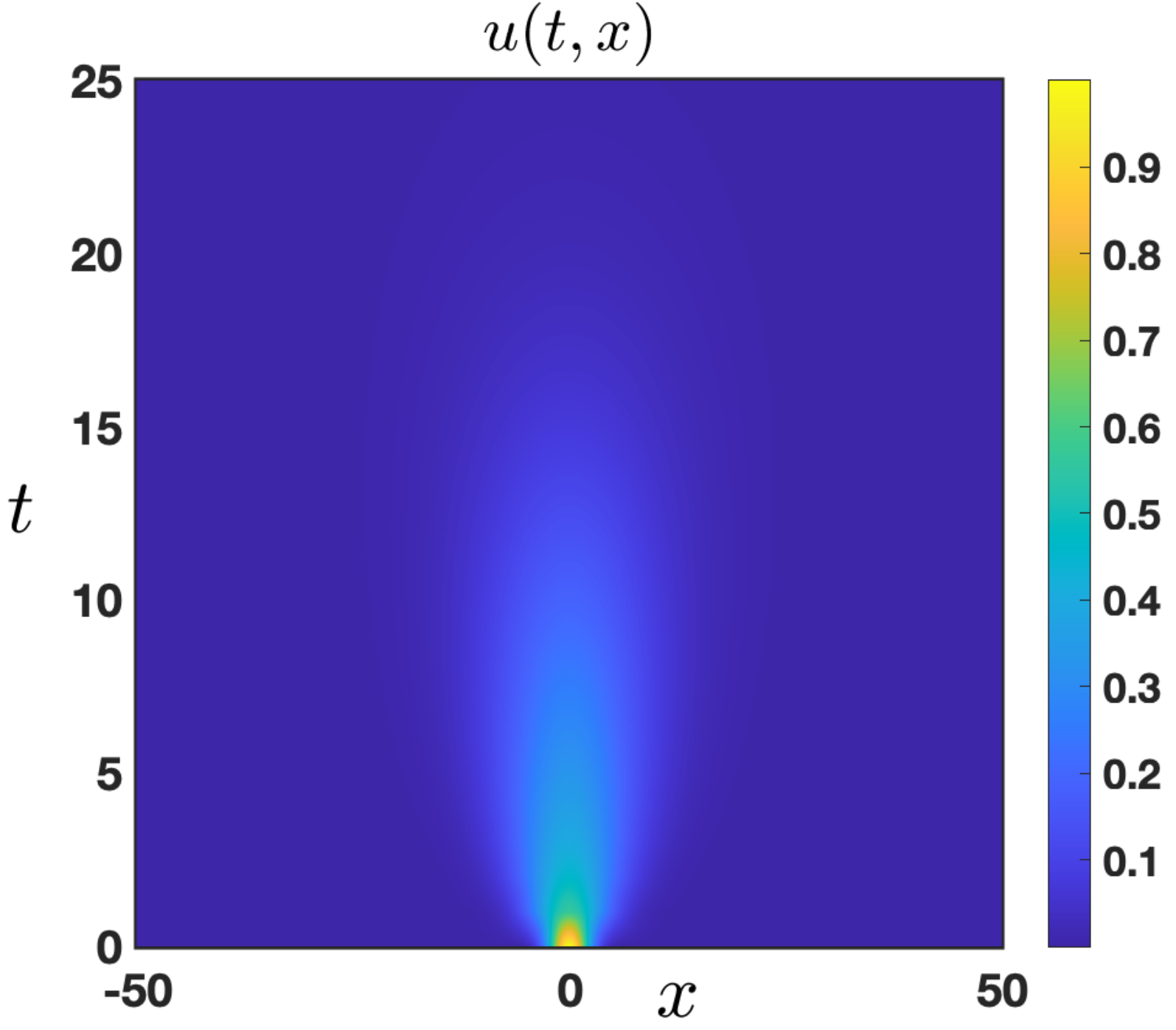}}\hspace{0.1in}
   \subfigure[Stagnation.]{\includegraphics[width=.29\textwidth]{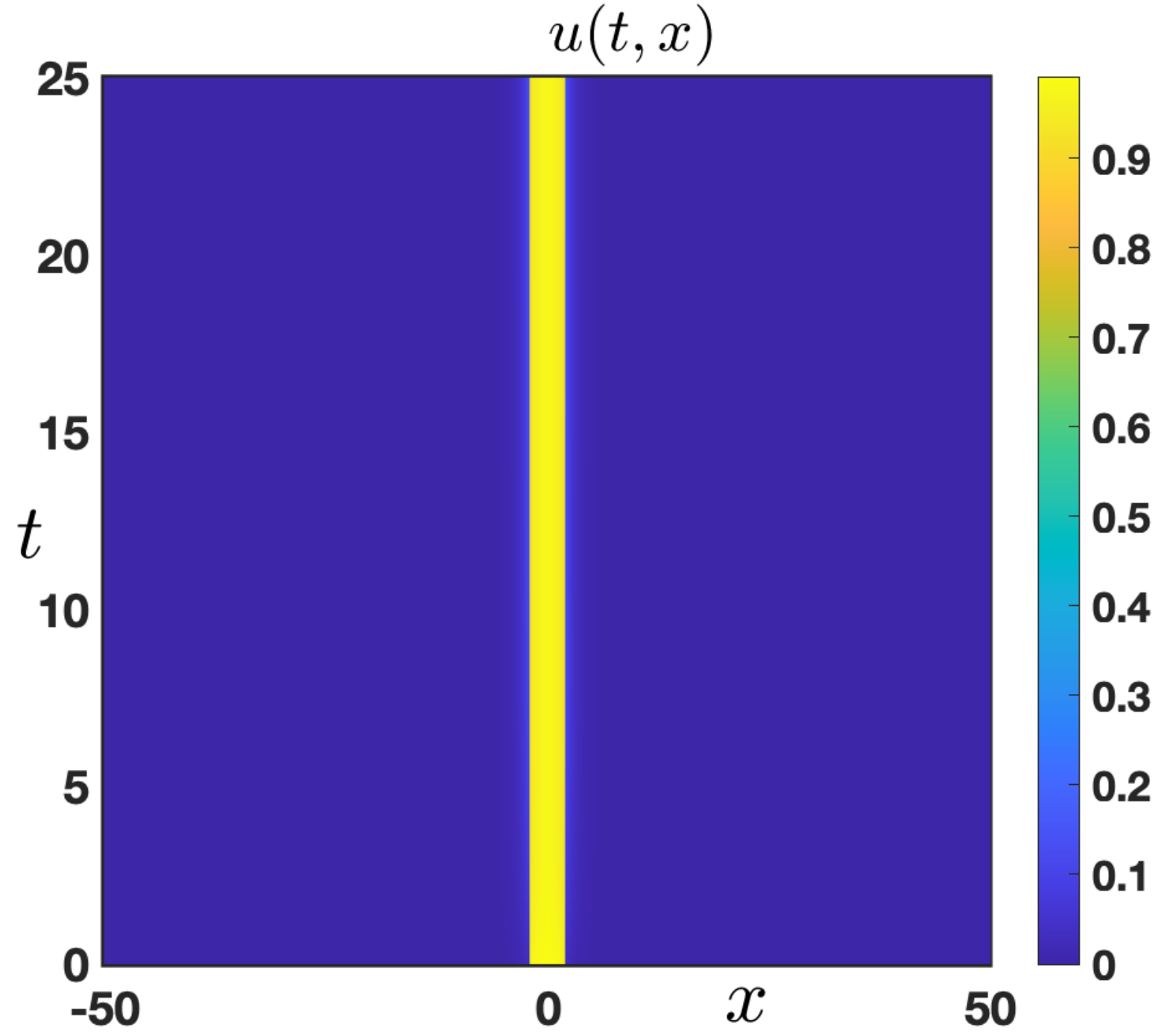}}\hspace{0.1in}
   \subfigure[Propagation.]{\includegraphics[width=.3\textwidth]{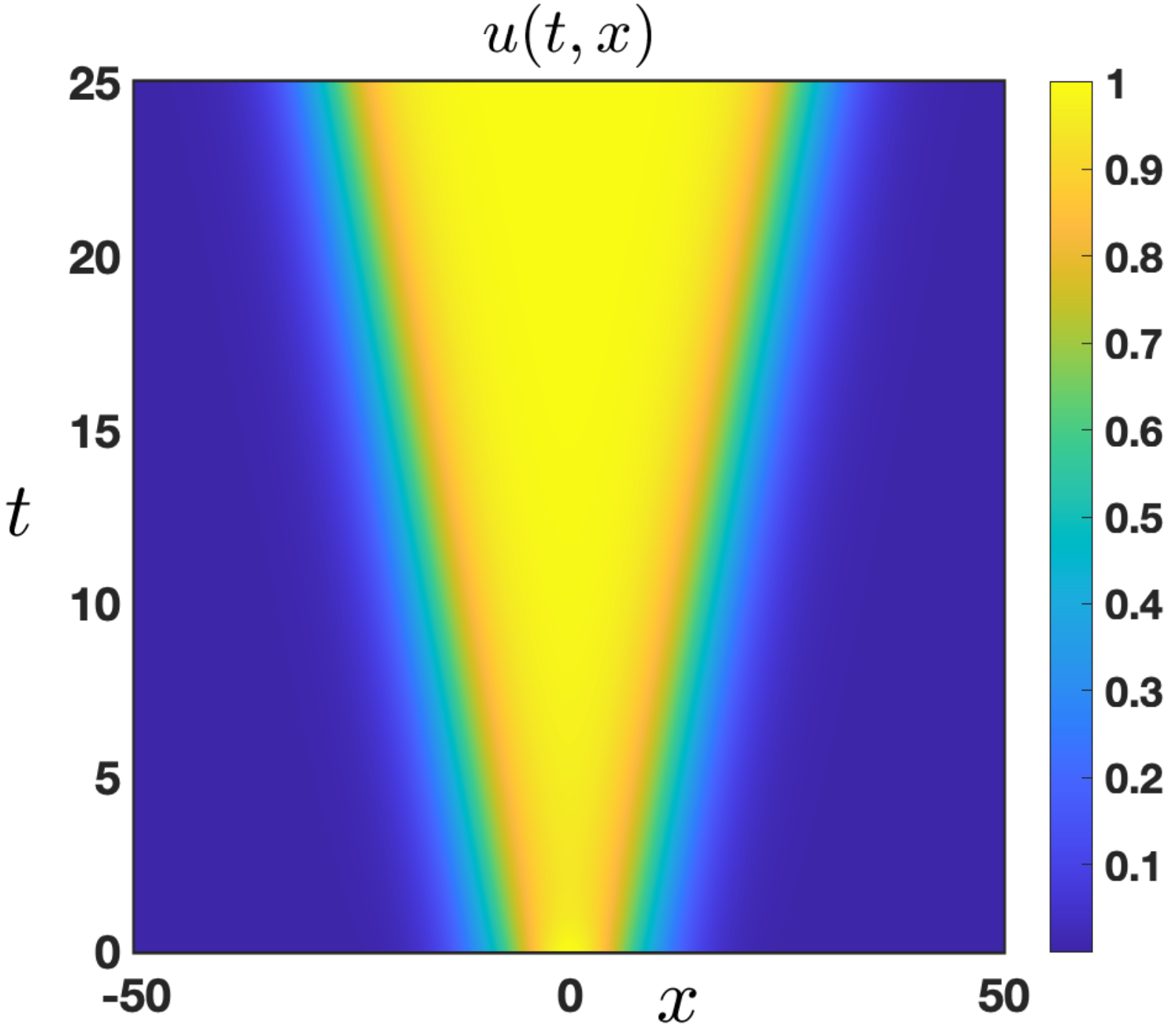}}
  \caption{Illustration of the different possible asymptotic behaviors of the solutions of the Cauchy problem~\eqref{edp}. }
  \label{fig:asymptotic}
\end{figure}

Before proceeding with an informal summary of our main results, we properly define what we shall refer to as extinction, propagation, and stagnation in the rest of the paper.  We refer to Figure~\ref{fig:asymptotic} for an illustration of these different asymptotic behaviors.

\begin{defi}Assume that $f$ and $\K$ satisfy Hypothesis~\ref{hypf} and~\ref{hypK} respectively. Let $\ell>0$ and $u_\ell\in\mathscr{C}^1\left((0,+\infty),L^1(\R)\cap L^\infty(\R) \right)\cap \, \mathscr{C}^0\left([0,+\infty),L^1(\R)\cap L^\infty(\R) \right)$ be the solution of the Cauchy problem \eqref{edp}. 
\begin{itemize}
\item {\bf Extinction:} If $u_\ell$ satisfies
\bqs
u_\ell(t,x)\underset{t\to +\infty}{\longrightarrow} 0 \text{ uniformly in } x\in\R,
\eqs
then we say that there is extinction for that given value of $\ell>0$.
\item {\bf Propagation:} If $u_\ell$ satisfies
\bqs
u_\ell(t,x)\underset{t\to +\infty}{\longrightarrow} 1 \text{ locally uniformly in } x\in\R,
\eqs
then we say that there is propagation for that given value of $\ell>0$.
\item {\bf Stagnation:} If $u_\ell$ satisfies
\bqs
u_\ell(t,x)\underset{t\to +\infty}{\longrightarrow} U(x) \text{ uniformly in } x\in\R\backslash\{\pm\ell\},
\eqs
where $U>0$ is an even stationary ground state solution of \eqref{edp} verifying $0=d(-U+\K*U)+f(U)$ with $U(\pm\infty)=0$ with possible discontinuity points at $x=\pm\ell$,  then we say that there is stagnation for that given value of $\ell>0$.
\end{itemize}
\end{defi}

\paragraph{Informal summary of main results.} A natural question that we shall address in this present work is to understand if such sharp thresholds of propagation persists when the diffusion is no longer local but takes the form of nonlocal dispersal mechanism as expressed in \eqref{edp}. In a nutshell, our results show that when the diffusion coefficient is large enough, then similar threshold phenomena are expected than in the local setting. More precisely, we prove that extinction always occurs if the diffusion coefficient is large enough (see Proposition~\ref{propext}), and that propagation occurs whenever there exists a smooth traveling front solution with non zero wave speed (see Proposition~\ref{proppropa}).  On the other hand, for small diffusion coefficient, the picture is totally different from the local case. It is possible to find regimes where extinction never occurs (see Proposition~\ref{prop51} and Proposition~\ref{prop52}). It is also possible to exhibit regimes where propagation never occurs (see Proposition~\ref{prop53}). And finally, lack of both propagation and extinction can be shown to happen for values of $\ell$ in intervals rather than in singleton as in the local case (see Corollary~\ref{cor}). The lack of propagation is often referred to as {\em pinning} in the literature.   Such propagation failure was first studied by Bates et al. \cite{BFRW} by proving the existence of discontinuous stationary fronts connecting $u=1$ to $u=0$ for bistable nonlocal reaction-diffusion equations under  Hypothesis~\ref{hypf} and Hypothesis~\ref{hypK}. Such discontinuous stationary interfaces turn out to be stable for the dynamics against bounded perturbations \cite{BFRW}. This is in sharp contrast with the local case where stationary interfaces, which are necessarily smooth, are always unstable due to the presence of an unstable eigenvalue of the linearized operator. The transition between pinning and unpinning in this nonlocal context was investigated more recently by Anderson et al. \cite{AFSS} who proved general universal asymptotic expansions of the wave speed of the interfaces. Ground state solutions, which are stationary even positive solutions asymptotic to $u=0$ at infinity, will be central in proving some of our results and we dedicate efforts at deriving quantitative properties of such solutions (see Section~\ref{secGS}).

\paragraph{Disclaimer.} When finalizing the current paper, we have been aware of two preprints \cite{ADK22,ZL22}  establishing very close results to ours on the long time dynamics of solutions of the Cauchy problem~\eqref{edp}. In~\cite{ZL22}, the authors prove that there is a sharp threshold between extinction and propagation for compactly supported kernels, and under an assumption on the nonlinearity which translates in our setting to the large diffusion regime of Section~\ref{seclarge} below. On the other hand, the purpose of \cite{ADK22} is to derive quantitative estimates on the thresholds of propagation and extinction in the regime where the initial condition is slightly above the unstable steady state of the nonlinearity and relate these thresholds to the properties of the interaction kernels. Actually, \cite{ADK22} also derives propagation results more general than our Proposition~\ref{proppropa} (and that of \cite{ZL22}) since they do not require exponential localization of the kernel, but only some integrability condition of the traveling front at its end states. In order to have a self-content presentation, we decided to keep our results in the regime of large diffusion with somehow simpler proofs since we need the estimates on the parameters derived there. The main novelty and originality of our work is also to investigate the regime of small diffusion where stagnation can occur, which was left aside in \cite{ADK22,ZL22}.

\paragraph{Outline of the paper.} In Section~\ref{seclarge}, we prove that extinction always occurs if the diffusion coefficient is large enough, and that propagation occurs whenever there exists a smooth traveling front solution with non zero wave speed. In Section~\ref{secGS}, we study ground state solutions, which are bounded stationary solutions, possibly discontinuous, for a specific interaction kernel. Next in Section~\ref{secsmall}, we prove that in the regime of small diffusion, extinction cannot happen. We also demonstrate that propagation cannot happen in some sub-region of the parameters within the so-called pinning region. Finally, in Section~\ref{secnum}, we present a numerical exploration of the problem.

\section{Regime of large diffusion}\label{seclarge}

\subsection{Extinction for large diffusion}\label{secext}

In this subsection, we prove that extinction occurs if the diffusion coefficient is large enough. This is summarized in the following result.

\begin{prop}[Extinction]\label{propext} Assume that $f$ and $\K$ satisfy Hypothesis~\ref{hypf} and Hypothesis~\ref{hypK}  respectively. Let $\kappa:=\underset{u\in(a,1)}{\sup}\frac{f(u)}{u}>0$. If $d>\kappa$, then there exists $\ell_0>0$ such that for all $\ell \in(0,\ell_0)$ the solution $u_\ell$ of \eqref{edp} satisfies
\bqs
u_\ell(t,x)\underset{t\to +\infty}{\longrightarrow} 0 \text{ uniformly in } x\in\R.
\eqs
\end{prop}

\begin{proof}
We consider the following linear problem
\bqq
\label{lin}
\left\{
\begin{split}
\partial_t v &= d\left(-v+\K *v\right)+\kappa v, \quad t>0 \quad x\in\R, \\
v(t=0,x) & = \mathds{1}_{[-\ell,\ell]}(x), \quad x\in\R,
\end{split}
\right.
\eqq
whose solution $v\in\mathscr{C}^1\left((0,+\infty),L^1(\R)\cap L^\infty(\R) \right)\cap \, \mathscr{C}^0\left([0,+\infty),L^1(\R)\cap L^\infty(\R) \right)$ satisfies
\bqs
0\leq v(t,x) = e^{(\kappa-d)t}\mathds{1}_{[-\ell,\ell]}(x) +\int_0^t e^{(\kappa-d)(t-s)}\K*v(s,x)\md s, \quad t>0 \quad x\in\R.
\eqs
By positivity of $v$ and the fact that $\int_\R \K(x)\md x =1$, we have that
\bqs
\| v(t) \| _{L^1(\R)}= \| v(t=0) \| _{L^1(\R)} e^{\kappa t}= 2 \ell e^{\kappa t}, \quad t>0,
\eqs
by integration of \eqref{lin}. As a consequence, using the boundedness of $\K$, we get that
\bqq
v(t,x) \leq e^{(\kappa-d)t}  + \frac{2\ell \|\K \|_{L^\infty(\R)}}{d}e^{\kappa t},  \quad t>0 \quad x\in\R.
\label{intlin}
\eqq
We define
\bqs
t_0 := \frac{1}{\kappa-d} \ln \frac{a}{2}>0, \text{ and } \ell_0:= \frac{1}{2}\frac{d}{\|\K \|_{L^\infty(\R)}}e^{-d t_0}.
\eqs
Evaluating \eqref{intlin} at $t=t_0$, we obtain that
\bqs
0 \leq v(t_0,x) \leq a, \quad x\in \R.
\eqs
The final step of the proof consists in checking that the solution 
\bqq
\label{nonlinconstant}
\left\{
\begin{split}
\partial_t w &= d\left(-w+\K *w\right)+f(w), \quad t>0 \quad x\in\R, \\
w(t=0,x) & = \gamma , \quad x\in\R,
\end{split}
\right.
\eqq
with $0<\gamma<a$ satisfies 
\bqs
w(t,x)\underset{t\to +\infty}{\longrightarrow} 0 \text{ uniformly in } x\in\R.
\eqs
This easily follows from the fact that the solution $w$ of the above Cauchy problem \eqref{nonlinconstant} is constant in space and that the solution of the ODE problem $w'=f(w)$ with $w(0)=\gamma\in(0,a)$ converges asymptotically to zero because of the bistability property of $f$.

To conclude the proof, one let $\ell \in(0,\ell_0)$ and uses the solution $v(t,x)$ of \eqref{lin} as a super-solution to obtain that 
\bqs
0 \leq u_\ell(t_0,x)\leq  v(t_0,x) \leq \frac{a}{2}+ \frac{a}{2} \frac{\ell}{\ell_0}<a,  \quad x\in \R.
\eqs
The uniform convergence to zero then follows by comparing the solution for $t\geq t_0$ to the problem \eqref{nonlinconstant}.
\end{proof}

As a consequence of the above Proposition~\ref{propext}, when $d>\kappa$ we can define 
\bqs
0< \ell_0^*:= \sup\left\{ \ell>0 ~|~  \underset{t\to+\infty}{\lim} u_\ell(t,\cdot)=0 \text{ uniformly in } \R \right\}.
\eqs
When the nonlinearity is given by the cubic $f_a(u)=u(1-u)(u-a)$ we have the explicit formula $\kappa=d_{\mathrm{ext}}(a):=\frac{(1-a)^2}{4}$ with $0<a<1/2$. In this setting, our numerics suggest that the condition $d>d_{\mathrm{ext}}(a)$ is sharp within the class of exponentially localized kernels. For the specific choice $\K(x)=e^{-|x|}/2$, we prove in the forthcoming section that when $d$ is small enough, extinction no longer occurs. 

\subsection{Propagation outside the pinning region}\label{secprop}

In this subsection, we prove that when the nonlocal equation \eqref{edp} has traveling front solution with non zero wave speed then propagation always occurs if the mass of the initial condition is large enough. Our method of proof relies on the one developed by Fife \& McLeod \cite{FMcL} in the case of local diffusion. In order to be able to apply their arguments, we need to ensure that traveling front solutions with non zero wave speed exponentially converge towards their end states. This is the technical reason for introducing the following assumption.

 \begin{hyp}\label{hypexploc}
 We suppose that the convolutional kernel $\K$ satisfies Hypothesis~\ref{hypK} and that there exists $\lambda_0>0$ such that 
 \bqs
 \int_\R e^{\lambda_0 |x|} \K(x)\md x <+\infty.
 \eqs
 \end{hyp}
 
By a traveling front solution, we refer to the couple $(U,c)$ with $c\in\R$ and  profile $U$ solution of
\bqq
\label{TF}
\left\{
\begin{split}
-cU' &= d\left(-U+\K *U\right)+f(U), \text{ on } \R, \\
U(-\infty) & = 1, \quad U(+\infty)=0,
\end{split}
\right.
\eqq
such that $u(t,x)=U(x-ct)$ is an entire solution to \eqref{edp}. Existence, uniqueness and qualitative properties of traveling front solution $(U,c)$ to \eqref{TF} with bistable nonlinearities has been investigated in various studies \cite{BFRW,CovD07,Chen,AFSS,BR17}, and we especially refer to \cite{BFRW} for the first results on the existence, uniqueness, stability and regularity of such traveling front solutions. In the proof of our main result, the regularity property of traveling front solutions will play an important role. As noticed in \cite{BFRW}, unlike the local case, discontinuous traveling front solutions may exist with zero wave speed $c=0$ even though $\int_0^1f(u)\mathrm{d} u>0$. Following \cite{BFRW}, a sufficient condition to ensure that the wave speed is non zero, and thus that the corresponding profile $U$ is smooth, is to further assume that the map $u\mapsto u - \frac{f(u)}{d}$ is strictly monotone on $[0,1]$. We will see that this map and its monotonicity properties will also play an important role in the forthcoming section when studying stationary ground state solutions (see Section~\ref{secGS}). Note that when the wave speed is non zero, then its sign is given by the sign of $\int_0^1f(u)\mathrm{d} u$, such that in our case, when non zero, the wave speed will always by positive. Using the result of \cite{CovD07}, under Hypothesis~\ref{hypexploc}, if there exists some traveling front $(U_*,c_*)$ solution of \eqref{TF} with $c_*>0$, then $U_*$ converges asymptotically towards its end states at $\pm\infty$ at an exponential rate. 
 
 \begin{prop}[Propagation]\label{proppropa}
 Assume that $f$ and $\K$ satisfy Hypothesis~\ref{hypf} and Hypothesis~\ref{hypexploc}  respectively. Assume that there exists $(U_*,c_*)$ solution of \eqref{TF} with $c_*>0$. Then, there exists $\ell_1>0$ such that for all $\ell \geq \ell_1$, the solution $u_\ell$ of \eqref{edp} satisfies
 \bqs
 u_\ell(t,x) \underset{t\to+\infty}{\longrightarrow}1,
 \eqs
 locally uniformly on $\R$.
 \end{prop}
 
 \begin{proof}
 The idea is to construct a subsolution for \eqref{edp} in the form of the superposition of two-counter propagating traveling front solutions, as done in \cite{FMcL} in the local case. For completeness, we reproduce the argument here.  We want to construct a subsolution $\underline{u}(t,x)$ which can be written
 \bqs
 \underline{u}(t,x)=U_+(t,x)+U_-(t,x)-1-q(t),
 \eqs
 with $U_\pm(t,x):=U_*(\pm x-c_*t-\zeta(t))$ where $q(t)>0$ and $\zeta(t)>0$ will be chosen appropriately. First, we define the functional 
 \bqq
  \mathscr{N}( w(t,x)):=\partial_t  w(t,x) -d\left( w(t,x) +\K* w(t,x)\right)-f( w(t,x)),
 \label{eqNcal}
 \eqq
 for a given function $w\ \in\mathscr{C}^1\left((0,+\infty),L^1(\R)\cap L^\infty(\R) \right)\cap \, \mathscr{C}^0\left([0,+\infty),L^1(\R)\cap L^\infty(\R) \right)$. Then, we compute
 \bqs
 \mathscr{N}( \underline{u}(t,x))=-\zeta'(t)\left(U_*'(\xi_-(t,x))+U_*'(\xi_+(t,x)) \right)-q'(t)+f(U_+(t,x))+f(U_-(t,x))-f( \underline{u}(t,x)),
 \eqs
 where we denoted $\xi_\pm(t,x):=\pm x-c_*t-\zeta(t)$ and used the fact that $U_*$ is a traveling front solution. Our aim is to prove that $\mathscr{N}( \underline{u}(t,x))\leq 0$ for all $t>0$ and $x\in\R$. We are only going to prove that $\mathscr{N}( \underline{u}(t,x))\leq 0$ for all $t>0$ and $x\geq0$, the estimate for $x\leq 0$ being handled similarly.

 Let $q_0>0$ be small enough such that $a<1-q_0<1$. There exist $\mu>0$ and $\delta>0$ such that
 \bqs
 f(u)-f(u-p)\geq \mu p,
 \eqs
 for $1-\delta \leq u \leq 1$ and $0\leq p \leq q_0$. Here, we have used the fact that $f'(1)<0$. In the region where $1-\delta \leq U_+(t,x) \leq 1$ and $0\leq 1-U_-(t,x)+q(t) \leq q_0$ we have that
 \bqs
 f(U_+(t,x))-f(U_+(t,x)-(1-U_-(t,x)+q(t))) \leq -\mu (1-U_-(t,x)+q(t)).
 \eqs
 Let us remark that the inequality $0\leq 1-U_-(t,x)+q(t) \leq q_0$ will hold if $0\leq q(t)\leq \frac{q_0}{2}$ and $\zeta(t)$ large enough. Indeed, if $\zeta(t)$ is large enough then $1-U_-(t,x) \leq \frac{q_0}{2}$ and $0\leq 1-U_-(t,x)+q(t) \leq q_0/2+q_0/2=q_0$. Finally, let us note that $U'_*(\xi_\pm(t,x))<0$ and that one can find a constant $b>0$ such that $f(U_-(t,x))\leq b(1-U_-(t,x))$. Summarizing, we see that for $t>0$, $x\geq 0$, $1-\delta \leq U_+(t,x) \leq 1$, $0\leq q(t)\leq \frac{q_0}{2}$ and $\zeta(t)$ large enough one has
 \bqs
  \mathscr{N}( \underline{u}(t,x))\leq (b-\mu) (1-U_-(t,x))-q'(t)-\mu q(t),
 \eqs
 provided that $\zeta'(t)<0$. We now use the fact $U_*$ asymptotically converges towards its rest states at an exponential rate to get the existence of $\nu>0$ and $\alpha>0$ such that
 \bqs
 1-U_-(t,x) \leq \alpha e^{\nu(-x-c_*t-\zeta(t))},
 \eqs
 for $\zeta(t)$ large enough. Then, we let $0<\mu_0<\min(\mu,\nu c_*)$ and $q(t) = \frac{q_0}{2}e^{-\mu_0 t}$. With that specific choice for $q(t)$, we get that
 that for $t>0$, $x\geq 0$, $1-\delta \leq U_+(t,x) \leq 1$ and $\zeta(t)$ large enough with $\zeta'(t)<0$ one has
 \bqs
  \mathscr{N}( \underline{u}(t,x))\leq 0.
 \eqs
 A similar argument also holds in the range $0\leq U_+(t,x)\leq \delta$ and $0\leq q(t)\leq \frac{q_0}{2}$ using this time the fact that $f'(0)<0$. In the range where $\delta \leq U_+(t,x) \leq 1-\delta$, we use the facts that $U_*'(\xi_-(t,x))+U_*'(\xi_+(t,x))\leq -\beta <0$ for some $\beta>0$ and
 \bqs
 f(U_+(t,x))-f(U_+(t,x)-(1-U_-(t,x)+q(t))) \leq C (1-U_-(t,x)+q(t)),
 \eqs
 for some constant $C>0$. We obtain with $q(t)=\frac{q_0}{2}e^{-\mu_0 t}$
  \begin{align*}
  \mathscr{N}( \underline{u}(t,x))&\leq \beta \zeta'(t) -q'(t)+b(1-U_-(t,x))+C (1-U_-(t,x)+q(t)),\\
  &\leq \beta \zeta'(t) + C_1 e^{-\mu_0 t} + C_2 e^{-\nu c_*t},
 \end{align*}
 for two positives constants $C_{1,2}>0$. We set $\zeta(t)$ solution of 
 \bqs
 \beta \zeta'(t) + C_1 e^{-\mu_0 t} + C_2 e^{-\nu c_*t}=0,
 \eqs
 starting from $\zeta(0)=\zeta_0>0$. We can select $\zeta_0>0$ large enough such that the inequalities
\bqs
 1-U_-(t,x) \leq \alpha e^{\nu(-x-c_*t-\zeta(t))} \leq \frac{q_0}{2}
 \eqs
 hold for all $t>0$.
 
 The last part of the proof now consists in being able to compare the solution $u_\ell$ with the subsolution $\underline{u}(t,x)$. More precisely, we will have that $ \underline{u}(t,x)$ will be a subsolution if we can ensure that $\underline{u}(t=0,x)\leq \mathds{1}_{[-\ell,\ell]}(x)$ for all $x\in\R$. At time $t=0$, we have
 \bqs
 \underline{u}(t=0,x)=U_*(x-\zeta_0)+U_*(-x-\zeta_0)-1-\frac{q_0}{2} \leq 1-\frac{q_0}{2}\leq 1, \quad \forall x\in[-\ell,\ell].
 \eqs
 Furthermore, there exists $\ell_1>0$ such that for all $|x|\geq \ell_1$ there holds $U_*(x-\zeta_0)+U_*(-x-\zeta_0)-1-\frac{q_0}{2} \leq0$ since $U_*(+\infty)=0$. As a consequence, for all $\ell\geq \ell_1$ we have $\underline{u}(t=0,x)\leq \mathds{1}_{[-\ell,\ell]}(x)$ for all $x\in\R$. Then the comparison principle gives that
 \bqs
U_*(x-c_*t-\zeta(\infty))+U(-x-c_*t -\zeta(\infty))-1-q(t) \leq \underline{u}(t,x) \leq u_\ell(t,x), \quad t>0, \quad x\in\R,
 \eqs
 and the local uniform convergence follows.
 \end{proof}
  
 As a consequence of the above Proposition~\ref{proppropa}, we can define 
\bqs
0< \ell_1^*:= \inf\left\{ \ell>0 ~|~  \underset{t\to+\infty}{\lim} u_\ell(t,\cdot)=1 \text{ locally uniformly in } \R \right\},
\eqs
whenever the assumptions of that proposition are satisfied.  Using the analysis in \cite{BFRW}, it is possible in the case of the cubic nonlinearity $f_a(u)=u(1-u)(u-a)$ to derive an explicit condition on $(d,a)$ for the existence of traveling front solutions with non zero wave speed. It is given in \cite{AFSS} and reads
\bqs
d>d_{\mathrm{pin}}(a):=\frac{1}{3}\left( 1-a+a^2-\sqrt{1-2a}\right), \quad 0<a<1/2.
\eqs
Combining this bound with the one obtained in the case of extinction in the previous section, we have proved the existence of a threshold phenomenon for \eqref{edp} whenever $d>\max\left\{d_{\mathrm{pin}}(a),d_{\mathrm{ext}}(a)\right\} $ in the case of the cubic nonlinearity $f_a(u)=u(1-u)(u-a)$ with $0<a<1/2$ that is
\bqs
0<\ell_0^*\leq \ell_1^*<+\infty.
\eqs
Our numerics suggest that this threshold is {\em sharp} when $d\geq \frac{1}{4}$ that is 
\bqs
0<\ell_0^*= \ell_1^*<+\infty, \quad \text{ if } d\geq \frac{1}{4} \text{ and } 0<a<1/2,
\eqs
whereas it is not sharp when $d< \frac{1}{4}$, that is
\bqs
0<\ell_0^*< \ell_1^*<+\infty, \quad \text{ if } \max\left\{d_{\mathrm{pin}}(a),d_{\mathrm{ext}}(a)\right\}<d< \frac{1}{4} \text{ and } 0<a<1/2.
\eqs
We refer to Section~\ref{secnum} for further details.

\section{Ground states -- A case study}\label{secGS}

Ground states for \eqref{edp} are bounded stationary profiles $U$ in $L^2(\R)$ solutions of 
\bqq
\label{GS}
\left\{
\begin{split}
0 &= d\left(-U+\K *U\right)+f(U), \text{ on } \R, \\
U(\pm\infty) & = 0, \quad U>0 \text{ on } \R.
\end{split}
\right.
\eqq
Note that the existence of solutions of \eqref{GS} for general nonlinearity $f$ and kernel $\K$ satisfying  Hypothesis~\ref{hypf} and Hypothesis~\ref{hypK} was proved by Chmaj \& Ren \cite{CR99} using variational techniques by studying the following energy functional defined on $L^2(\R)$
\bqs
\mathscr{E}(u):=\frac{d}{4}\int_{\R^2}\K(x-y)(u(x)-u(y))^2\md x \md y + \int_\R F(u(x))\md x,
\eqs
with $F(u):=-\int_0^u f(v)\md v$. It is also not difficult to check that \eqref{edp} is the $L^2$ gradient flow of the above functional. Depending on $d$ and the nonlinearity $f$, it was already noticed in \cite{CR99} that solutions of \eqref{GS} can be discontinuous, see also \cite{AFSS,BFRW}. This is in total contrast with the local diffusion case where it is well-known \cite{BL83} that ground states are smooth whenever the nonlinearity $f$ is. In what follows, we will go beyond the results of existence of Chmaj \& Ren \cite{CR99} and completely characterize all possible ground states of \eqref{GS} in the case where the kernel $\K$ is fixed to $\K(x)=e^{-|x|}/2$ and for any bistable nonlinearity satisfying our running Hypothesis~\ref{hypf} and such that the function $g$ defined below is monotone on a maximum of 3 intervals.  This specific choice for $\K$ will enable us to recast our nonlocal problem to local one such that phase plane analysis techniques can be used. We apply our results to the cubic $f_a(u)=u(1-u)(u-a)$ with $0<a<1/2$ for which we can  derive closed form formulas.  We finally note that due to the symmetry of $\K$ and the translation invariance of \eqref{GS}, we only look for ground states that are even (possibly discontinuous) functions on $\R$.


\subsection{Reformulation of the problem when $\K(x)=e^{-|x|}/2$}

As explained above, we can reformulate \eqref{GS} as a higher order dynamical system when $\K(x)=e^{-|x|}/2$.  We set $V=\K*U$ and remark that our specific choice of $\K$ implies that the equality is equivalent to $V-V''=U$. This can be easily seen by noticing that the Fourier symbol of $\K$ is $\widehat{\K}(\xi)=(1+\xi^2)^{-1}$ for $\xi\in\R$. As a consequence, the nonlocal equation on $U$ can be written as
\bqs
\left\{
\begin{split}
0 &= d\left(-U+V \right)+f(U), \\
V-V''&=U,
\end{split}
\right.
\quad \text{ on }\R.
\eqs

We denote $g(u):=u-\frac{f(u)}{d}$ defined on $\R$ which has the same regularity as $f$ and satisfies $g(0)=0$ and $g(1)=1$ with $g'(0)>0$ and $g'(1)>0$. When $g'>0$ on $[0,1]$, then the first equation of the above system gives that $U=g^{-1}(V)$ such that we are let to study the second order dynamical system
\bqs
V-V''=g^{-1}(V), \text{ on } \R,
\eqs
subject to $V(\pm\infty)=0$ and $V>0$. In that case, phase plane analysis \cite{BL83,CR99} gives the existence of a unique, smooth, even, ground state solution $V$ which is monotone on $\R^+$.  Indeed, the above system is Hamiltonian with conserved Hamiltonian
\bqs
\mathcal{H}(V,V')=\frac{1}{2}V'^2+\mathcal{G}(V), \text{ with } \mathcal{G}(V):=-\frac{1}{2}V^2+\int_0^V g^{-1}(w) \md w.
\eqs
The ground state $V$ is then found as the unique homoclinic solution parametrized by
\bqs
\mathcal{H}(V,V')=0.
\eqs
Let us remark that the condition $g'>0$ on $[0,1]$ is always satisfied when $d$ is large enough. This can be summarized in the following result.

\begin{lem}\label{lem41}
Assume that $d>\| f' \|_{\infty}$,  then there exists a unique, smooth, even, ground state solution $V$ of \eqref{GS} which further satisfies $V'(x)<$ for all $x>0$.
\end{lem}

When $g$ is no longer invertible on $[0,1]$, it is possible that it may have several intervals of strict monotonicity. Since $g'(0)>0$ and $g'(1)>0$, the function $g$ is at least strictly monotone on two intervals near $u=0$ and $u=1$. In what follows, we assume that $g$ is strictly monotone on a maximum of 3 intervals. A consequence of this assumption will be that discontinuous ground states of \eqref{GS} will have only two points of discontinuity at $\pm x_0$ for some $x_0$.

\begin{hyp}\label{hypg}
We assume that $g(u):=u-\frac{f(u)}{d}$, defined on $\R$, is monotone on a maximum of 3 intervals on $[0,1]$. In other words, there exist $0<\beta \leq \gamma<1$ such that
\bqs
g'(u)>0 \text{ for } u \in [0,\beta) \cup (\gamma,1], \quad  g'(u)<0 \text{ for } u \in (\beta,\gamma),
\eqs
with $g'(\beta)=g'(\gamma)=0$.
\end{hyp}

\begin{figure}[t!]
  \centering
  \includegraphics[width=.475\textwidth]{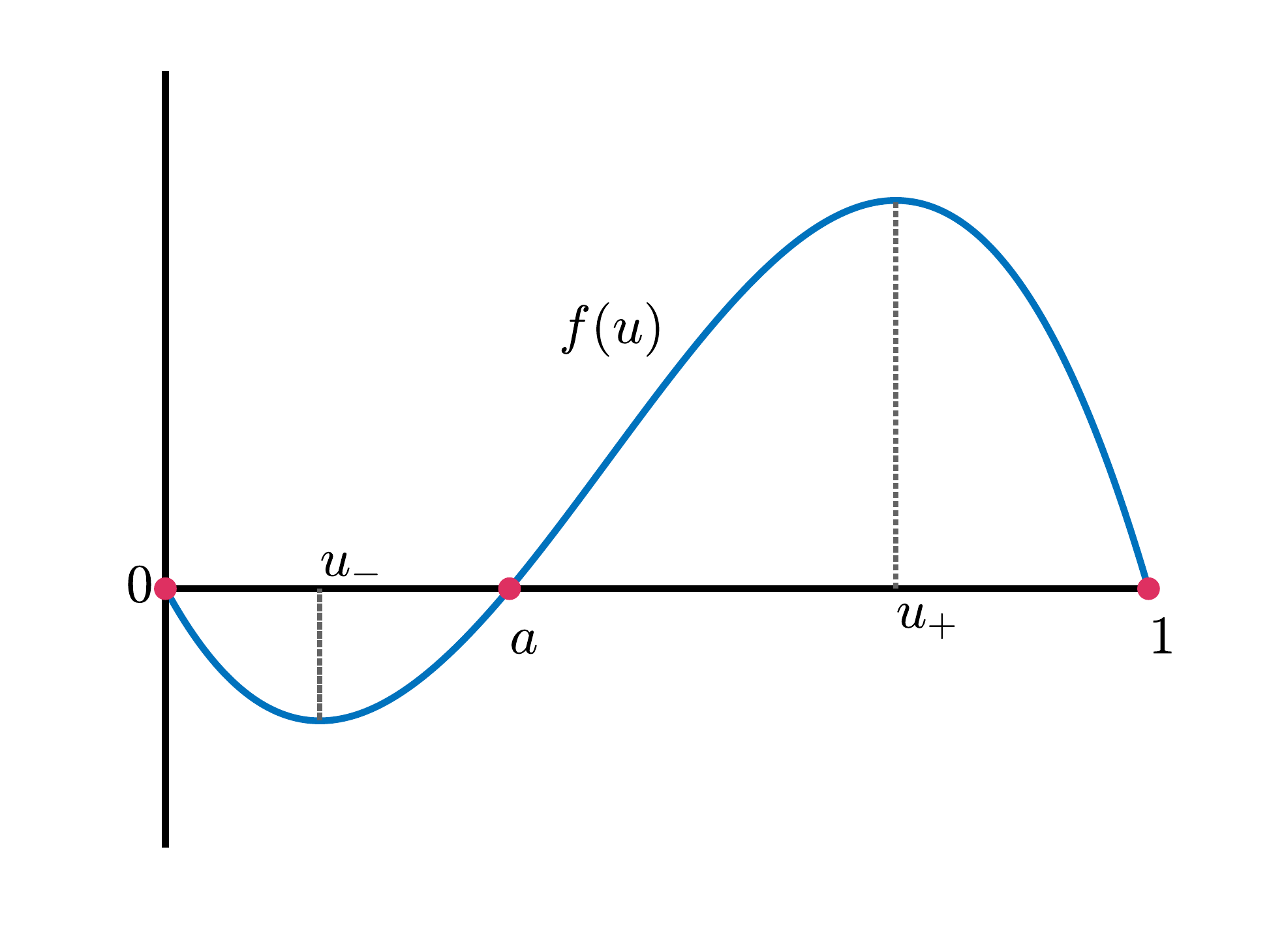}
  \includegraphics[width=.475\textwidth]{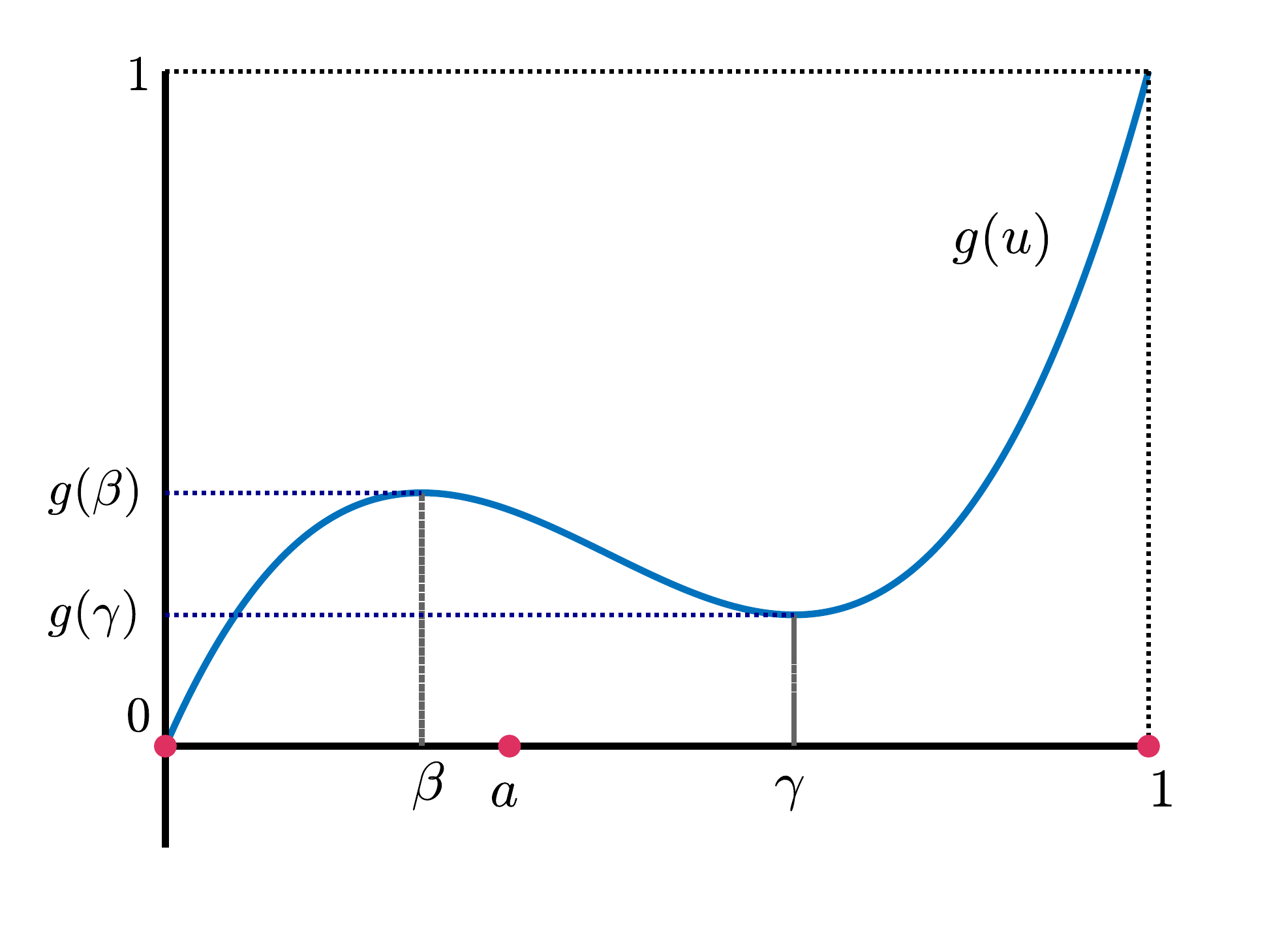}
  \caption{We illustrate a typical situation where both Hypothesis~\ref{hypf} and Hypothesis~\ref{hypg} are satisfied for $f$ and $g$ with $\beta<\gamma$.}
  \label{fig:figurefg}
\end{figure}

Under the above assumption, we denote by $g_{-}^{-1}$ and $g_{+}^{-1}$ the respective inverses of $g$ on $[0,\beta[$ and $]\gamma,1]$ such that we are eventually let to study two problems
 \bqs
V-V''=g_{-}^{-1}(V), \text{ and } V-V''=g_{+}^{-1}(V).
\eqs
To obtain a ground state, we will glue the trajectories of the solutions of each problem, and when $\beta<\gamma$, we typically expect to have discontinuous profiles. The case when $\beta=\gamma$ is degenerate and will be treated in a second time. So, from now on, we assume that $\beta<\gamma$.

The key point of our analysis is that both systems have a Hamiltonian which is conserved along the trajectory of the solution and is given by
\bqs
\mathcal{H}_\pm(V,V')=\frac{1}{2}V'^2+\mathcal{G}_\pm(V),
\eqs
where
\bqs
\mathcal{G}_-(V):=-\frac{1}{2}V^2+\int_0^V g_{-}^{-1}(w) \md w, \text{ and } \mathcal{G}_+(V):=-\frac{1}{2}V^2-\int_V^1 g_{+}^{-1}(w) \md w.
\eqs
Since ground states satisfy the asymptotic condition $V(\pm\infty)=0$, they will necessarily satisfy
\bqs
\mathcal{H}_-(V,V')=0,
\eqs
near infinity and this is the starting point of our analysis. We will distinguish different cases depending on the sign of $\mathcal{G}_-$ on its interval of definition. More precisely, since $g_{-}^{-1}$ is defined  on the interval $[0,g(\beta)[$, we have that $\mathcal{G}_-$ is also defined for all $V\in [0,g(\beta)[$, and we can proceed as follows.

\paragraph{Case $\beta<\gamma$ \& $\mathcal{G}_-(g(\beta))<0$.} If we suppose that $\mathcal{G}_-(g(\beta))<0$, then we necessarily have that
\bqs
\mathcal{G}_-(V) < 0, \text{ for all } V\in(0,g(\beta)],
\eqs
and the zero level set of $\mathcal{H}_-$ in the half-plane $v\geq0$, denoted $\mathcal{M}_0^-$,  can be parametrized as
\bqs
\mathcal{M}_0^-:=\left\{ (v,w)\in\R^2 ~|~ w = \pm \sqrt{-2\mathcal{G}_-(v)}, \quad 0 \leq v \leq g(\beta) \right\}.
\eqs
We emphasis that $\mathcal{M}_0^-$ is composed of two symmetric curves in the half-plane $v\geq0$ intersecting at the origin. By definition of $\beta$ and $\gamma$, we always have $g(\gamma)<g(\beta)$ when $\beta<\gamma$. Thus, for each element $(v_0,w_0)\in\mathcal{M}_0^-$ with $g(\gamma)\leq v_0 \leq g(\beta)$ and $v_0>0$, we can look for an intersection with a trajectory coming from the system defined near $u=1$ through $g_+^{-1}$, that is we solve for
\bqs
\mathcal{H}_+(V,V')=\mathcal{H}_+(v_0,w_0)=\mathcal{G}_+(v_0)-\mathcal{G}_-(v_0).
\eqs

In order to describe the geometry of the above level sets, we first need to understand if $\mathcal{M}_0^-$ intersects $\mathcal{W}_1^+$ which is defined as the level set of $\mathcal{H}_+$ associated to the equilibrium point $(1,0)$ restricted to the half-plane $v\leq 1$. That is we solve for
\bqs
\mathcal{H}_+(v,w)=\mathcal{H}_+(1,0)=\mathcal{G}_+(1)=-\frac{1}{2}, \quad v\in[g(\gamma),1],
\eqs
which gives the parametrization 
\bqs
\mathcal{W}_1^+:=\left\{ (v,w)\in\R^2 ~|~ w = \pm \sqrt{2\left(\mathcal{G}_+(1)-\mathcal{G}_+(v)\right)}, \quad g(\gamma) \leq v \leq 1 \right\}.
\eqs
As a consequence, we need to characterize when the most extremal points on $\mathcal{M}_0^-$, which are given by $(g(\beta),\pm\sqrt{-2\mathcal{G}_-(g(\beta))})$, can belong to $\mathcal{W}_1^+$. That is, we look for solutions to
\bqs
\mathcal{G}_+(g(\beta))-\mathcal{G}_-(g(\beta))=\mathcal{G}_+(1)=-\frac{1}{2}.
\eqs
Using the specific form of $\mathcal{G}_\pm$, we can write the above equality as
\bqs
\int_0^{g(\beta)}g_-^{-1}(w)\md w + \int_{g(\beta)}^1g_+^{-1}(w)\md w=\frac{1}{2},
\eqs
and upon denoting $\tilde{\beta}:= g_+^{-1}(g(\beta))\in(\gamma,1)$, we get
\bqs
\int_0^\beta u g_-'(u)\md u + \int_{\tilde{\beta}}^1 u g_+'(u)\md u = \frac{1}{2}.
\eqs
And an integration by part gives:
\bqq
\int_0^\beta g_-(u)\md u + \int_{\tilde{\beta}}^1 g_+(u)\md u +(\tilde{\beta}-\beta)g(\beta)= \frac{1}{2}.
\label{boundaryPR}
\eqq
The relation \eqref{boundaryPR} is precisely the definition of the boundary of the so-called pinning region \cite{AFSS} (see also \cite{BFRW}) which characterizes the existence of discontinuous monotone traveling front solutions. We thus need to distinguish two sub-cases.

\begin{figure}[t!]
  \centering
 \subfigure[Phase plane analysis.]{ \includegraphics[width=.5\textwidth]{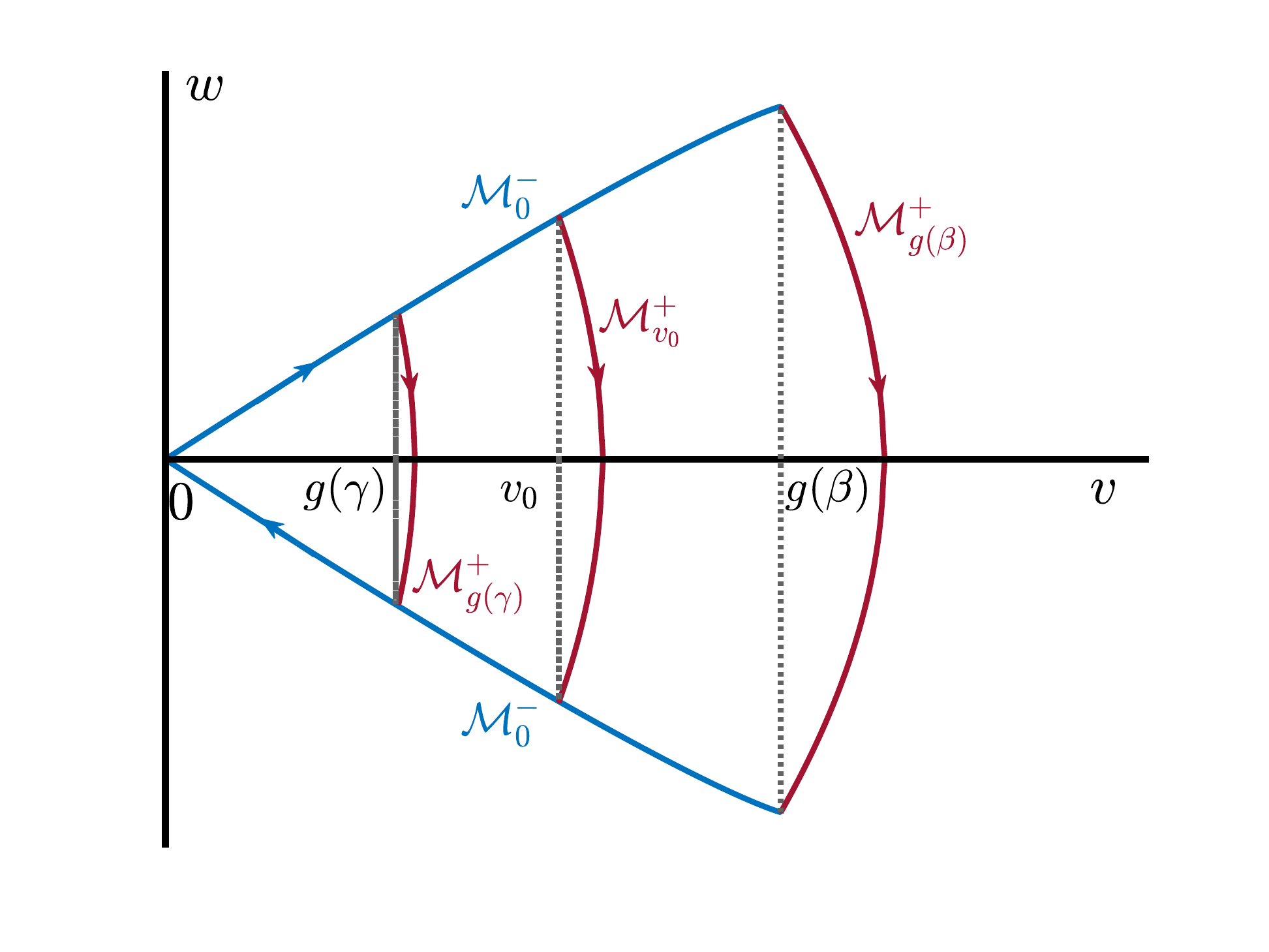}}
   \subfigure[Glued trajectory $V(x)$.]{   \includegraphics[width=.475\textwidth]{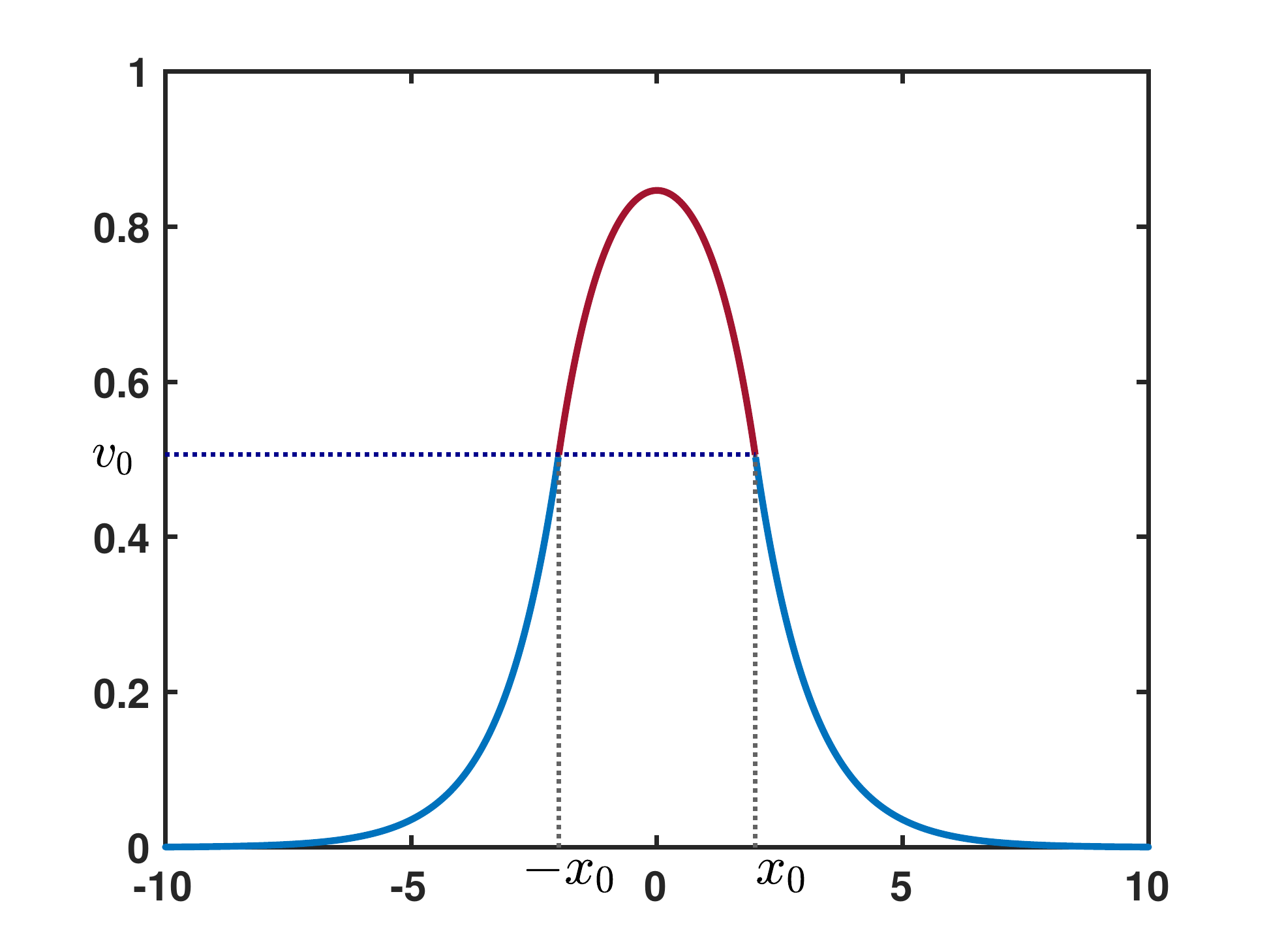}}
  \caption{Left: Illustration of the phase plane analysis and the gluing of trajectories in the case $\beta<\gamma$ \& $\mathcal{G}_-(g(\beta))<0$ with $g(\gamma)>0$ and $0<-\mathcal{G}_-(g(\beta))<\mathcal{G}_+(1)-\mathcal{G}_+(g(\beta))$. In blue, we have represented $\mathcal{M}_0^-$ which corresponds to the zero level set of $\mathcal{H}_-$ in the half-plane $v\geq 0$. In red, we have shown three trajectories joining the upper branch of $\mathcal{M}_0^-$ to its lower branch. There is a one-parameter family of parametrized curves $\mathcal{M}_{v_0}^+$ for each $v_0\in[g(\gamma),g(\beta)]$ which are bounded to the left by $\mathcal{M}_{g(\gamma)}^+$ and to the right by $\mathcal{M}_{g(\beta)}^+$. Right: Illustration of a glued trajectory leading to a solution $V(x)$ of \eqref{eq:SolV}. Here, the blue part of the solution corresponds to the trajectory on $\mathcal{M}_0^-$ while the red part is associated to the trajectory on $\mathcal{M}_{v_0}^+$. Note the continuity of the solution at the gluing point $\pm x_0$.}
  \label{fig:Case1}
\end{figure}

\begin{itemize}
\item Suppose that $0<-\mathcal{G}_-(g(\beta))<\mathcal{G}_+(1)-\mathcal{G}_+(g(\beta))$. In this case, for each $g(\gamma)\leq v_0 \leq g(\beta)$ and $v_0>0$, the level set $\mathcal{H}_+(v,w)=\mathcal{G}_+(v_0)-\mathcal{G}_-(v_0)$, in the half plane $v\geq v_0$, can be parametrized as
\bqs
\mathcal{M}_{v_0}^+:=\left\{ (v,w)\in\R^2 ~|~ w = \pm \sqrt{2\left(\mathcal{G}_+(v_0)-\mathcal{G}_-(v_0)-\mathcal{G}_+(v)\right)}, \quad v_0 \leq v \leq v_* \right\},
\eqs
where $v_*>v_0$ satisfies $\mathcal{G}_+(v_*)=\mathcal{G}_+(v_0)-\mathcal{G}_-(v_0)$ and 
\bqs
\mathcal{G}_+(v_0)-\mathcal{G}_-(v_0)-\mathcal{G}_+(v)>0, \text{ for all } v\in(v_0,v_*).
\eqs
We refer to Figure~\ref{fig:Case1} for an illustration of the geometrical configurations of $\mathcal{M}_0^-$ and $\mathcal{M}_{v_0}^+$ in the phase plane $(v,w)\in\R^2$ with $v\geq0$.

As a consequence, we have constructed a one parameter family of continuous solutions $V$  which are positive and even and satisfy the following properties. For each $v_0\in [g(\gamma),g(\beta)]$ with $v_0>0$, the function $V:\R \to [0,v_*]$ is defined as 
\bqq
\label{eq:SolV}
V(x) = \left\{ 
\begin{array}{lcl}
V_-(x) & \text{ if } & |x|\geq x_0,\\
V_+(x) & \text{ if } & |x|\leq x_0,
\end{array}
\right.
\eqq
where $V_-$ satisfies 
\bqs
\mathcal{H}_-(V_-(x),V_-'(x))=0, \quad |x|\geq  x_0, \text{ with } V_-(\pm \infty)=0,
\eqs
and $V_+$ satisfies
\bqs
\mathcal{H}_+(V_+(x),V_+'(x))=\mathcal{G}_+(v_0)-\mathcal{G}_-(v_0), \quad |x|\leq x_0, 
\eqs
with the continuity condition
\bqs
V_-(\pm x_0)=V_+(\pm x_0)=v_0, \text{ and } V_-'(\pm x_0)=V_+'(\pm x_0)=\pm \sqrt{-2\mathcal{G}_-(v_0)}.
\eqs
Here the point of continuity $x_0$ for $V$ is uniquely defined by
\bqq
x_0 = \int_{v_0}^{v_*} \frac{\md v}{\sqrt{2\left(\mathcal{G}_+(v_0)-\mathcal{G}_-(v_0)-\mathcal{G}_+(v)\right)}}>0.
\label{defx0}
\eqq
Indeed, by symmetry of $V_+$, we define $x_0>0$ such that
\bqs
\left\{ 
\begin{array}{l}
V_+(-x_0)=v_0,\\
V_+(0) =v_*.
\end{array}
\right.
\eqs
Then, using the strict monotony of $V_+$ on the interval $(-x_0,0)$ we remark that
\bqs
x_0=\int_{-x_0}^0 \md x =\int_{-x_0}^0 \frac{\md V_+(x)}{V_+'(x)} = \int_{v_0}^{v_*} \frac{\md v}{\sqrt{2\left(\mathcal{G}_+(v_0)-\mathcal{G}_-(v_0)-\mathcal{G}_+(v)\right)}}.
\eqs

Coming back to our original problem, for each $v_0\in [g(\gamma),g(\beta)]$ with $v_0>0$, the function $U:\R\to [0,g_+^{-1}(v_*)]$ defined out of $V_\pm$ as
\bqq
U(x) = \left\{ 
\begin{array}{lcl}
U_-(x)=g_-^{-1}(V_-(x)) & \text{ if } & |x|> x_0,\\
U_+(x)=g_+^{-1}(V_+(x)) & \text{ if } & |x|< x_0,
\end{array}
\right.
\label{solU}
\eqq
is a solution of \eqref{GS} with precisely two points of discontinuity at $\pm x_0$ where
\bqs
0<\underset{x\rightarrow x_0^+}{\lim} U(x) = g_-^{-1}(v_0) < g_+^{-1}(v_0) =\underset{x\rightarrow x_0^-}{\lim} U(x).
\eqs
We refer to Figure~\ref{fig:groundstateUV} for an illustration of ground state solution in that case and its relation to the solution $V$.
\item Suppose that $-\mathcal{G}_-(g(\beta)) \geq \mathcal{G}_+(1)-\mathcal{G}_+(g(\beta))$. In this case, there exists a unique $v_c\in (0,g(\beta)]\cap[g(\gamma),1)$ such that 
\bqs
\mathcal{G}_+(v_c)-\mathcal{G}_-(v_c) = \mathcal{G}_+(1),
\eqs
by monotonicity of the map $v\mapsto \mathcal{G}_+(v)-\mathcal{G}_-(v)$ on $(0,g(\beta)]\cap[g(\gamma),1)$. As a consequence, we can proceed along similar lines as in the previous case, and we can construct ground state solutions of the form \eqref{solU} for each $v_0\in[g(\gamma),v_c)$ with $v_0>0$. We refer to Figure~\ref{fig:CasePin} for a visualization of the phase plane configuration in that case.
\end{itemize}

\begin{figure}[t!]
  \centering
 \subfigure[Solution $V(x)$.]{   \includegraphics[width=.475\textwidth]{GroundStateV.pdf}}
  \subfigure[Ground state $U(x)$.]{   \includegraphics[width=.475\textwidth]{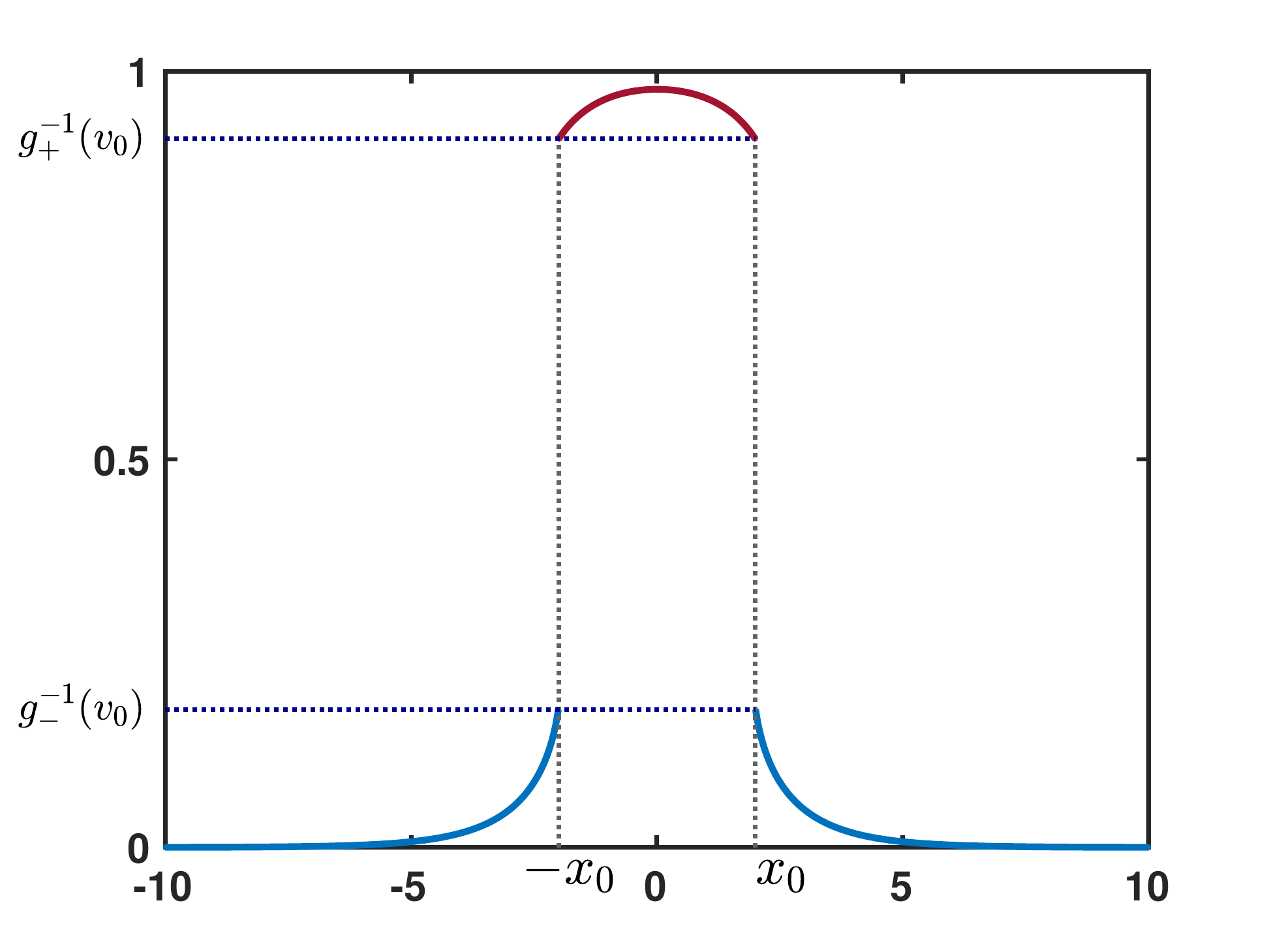}}
  \caption{Illustration of the relationship between the solution $V(x)$ and the ground state $U(x)$ in the setting of Lemma~\ref{lem43}. The profile $V$ is constructed by gluing the solutions of the two problems $V-V''=g_{-}^{-1}(V)$ (blue part of the solution) and $V-V''=g_{+}^{-1}(V)$ (dark red part of the solution) via a phase plane analysis. The ground state solution is then defined out of $V$ by using $g_{\pm}^{-1}$ on each branch as given in formula \eqref{solU}. Unlike $V$, the ground state solution $U$ has now two discontinuity points at $\pm x_0$.}
  \label{fig:groundstateUV}
\end{figure}

We have then proved the following result.

\begin{lem}\label{lem43}
Assume that $g$ satisfies Hypothesis~\ref{hypg}. Assume that $\beta<\gamma$ and $\mathcal{G}_-(g(\beta))<0$. Then, there exists a one parameter family of discontinuous ground state solutions of equation \eqref{GS} given by \eqref{solU} which are even, positive, with precisely two points of jump discontinuity and smooth otherwise. Furthermore, there is no other solution to \eqref{GS}.
\end{lem}

\begin{figure}[t!]
  \centering
 \includegraphics[width=.5\textwidth]{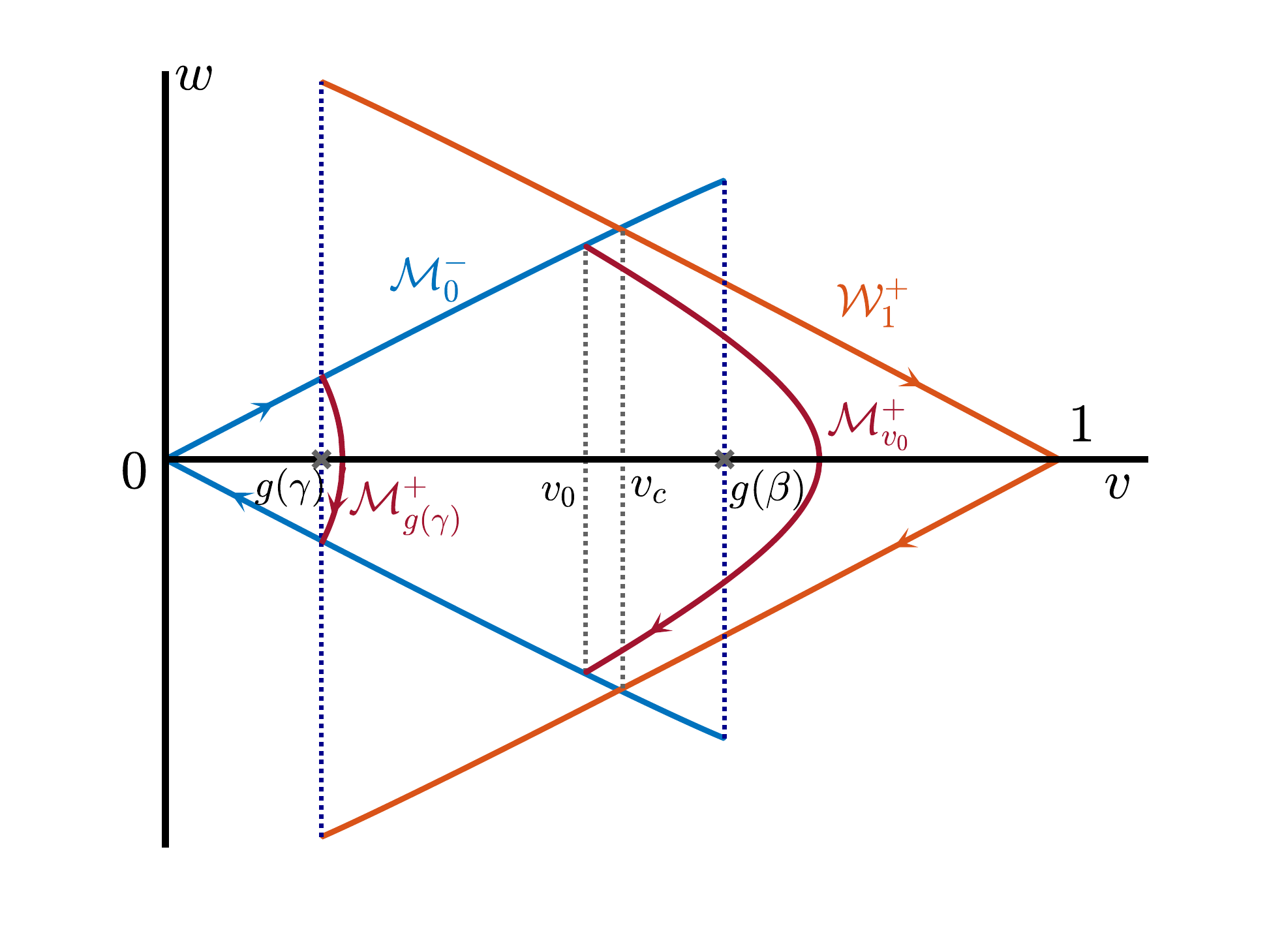}
  \caption{Illustration of the phase plane analysis and the gluing of trajectories in the case $\beta<\gamma$ \& $\mathcal{G}_-(g(\beta))<0$ with $g(\gamma)>0$ within the pinning region $-\mathcal{G}_-(g(\beta)) \geq \mathcal{G}_+(1)-\mathcal{G}_+(g(\beta))$. In blue, we have represented $\mathcal{M}_0^-$ which corresponds to the zero level set of $\mathcal{H}_-$ in the half-plane $v\geq 0$, while in orange we have represented $\mathcal{W}_1^+$ which corresponds to the level set of $\mathcal{H}_+$ associated to $(1,0)$. Note that $\mathcal{M}_0^-$ and $\mathcal{W}_1^+$ intersect at two points given by $(v_c,\pm\sqrt{-2\mathcal{G}_-(v_c)})$ with $v_c\in (g(\gamma),g(\beta))$. In red, we have shown two trajectories joining the upper branch of $\mathcal{M}_0^-$ to its lower branch. There is a one-parameter family of parametrized curves $\mathcal{M}_{v_0}^+$ for each $v_0\in[g(\gamma),v_c)$ which are bounded to the left by $\mathcal{M}_{g(\gamma)}^+$ and to the right by $\mathcal{W}_1^+$. }
  \label{fig:CasePin}
\end{figure}

\paragraph{Case $\beta<\gamma$ \& $\mathcal{G}_-(g(\beta))>0$.} In the case when $\mathcal{G}_-(g(\beta))>0$, we denote by $v_m $ the positive root of $\mathcal{G}_-$ in the interval $(0,g(\beta))$ that is $\mathcal{G}_-(v_m)=0$ and $\mathcal{G}_-(v)<0$ for all $v \in(0,v_m)$ with $v_m\in(0,g(\beta))$. As a consequence, the zero level set 
$\mathcal{H}_-$ in the half-plane $v\geq0$, still denoted $\mathcal{M}_0^-$,  can be parametrized as
\bqs
\mathcal{M}_0^-:=\left\{ (v,w)\in\R^2 ~|~ w = \pm \sqrt{-2\mathcal{G}_-(v)}, \quad 0 \leq v \leq v_m \right\}.
\eqs
The manifold $\mathcal{M}_0^-$ is a closed curve, since $\mathcal{G}_-(v_m)=0$, which is homoclinic to the origin. As a consequence, it gives the existence of a smooth, even, positive function $V_s:\R \to [0,v_m]$ solution of
\bqs
V-V''=g_{-}^{-1}(V), \text{ on } \R,
\eqs
which satisfies $V_s(\pm\infty)=0$. In turn, this gives the existence of a smooth, even, positive function $U_s:\R \to [0,g_-^{-1}(v_m)]$ solution to \eqref{GS} also verifying $U_s(\pm\infty)=0$ which is given by $U_s(x) = g_-^{-1}(V_s(x))$ for all $x\in\R$.

We now investigate if it is possible to construct discontinuous solutions following the same strategy as in the previous case. In order to be able to intersect $\mathcal{M}_0^-$ with a trajectory coming from the system defined near $u=1$ through $g_{+}^{-1}$, it is necessary that $g(\gamma)<v_m$, since in that case for any $0<v_0\in[g(\gamma),v_m]$, we can solve for
 \bqs
\mathcal{H}_+(V,V')=\mathcal{G}_+(v_0)-\mathcal{G}_-(v_0).
\eqs
The above level set, in the half plane $v\geq v_0$, can be once again parametrized as
\bqs
\mathcal{M}_{v_0}^+:=\left\{ (v,w)\in\R^2 ~|~ w = \pm \sqrt{2\left(\mathcal{G}_+(v_0)-\mathcal{G}_-(v_0)-\mathcal{G}_+(v)\right)}, \quad v_0 \leq v \leq v_* \right\},
\eqs
where $v_*>v_0$ satisfies $\mathcal{G}_+(v_*)=\mathcal{G}_+(v_0)-\mathcal{G}_-(v_0)$ and 
\bqs
\mathcal{G}_+(v_0)-\mathcal{G}_-(v_0)-\mathcal{G}_+(v)>0, \text{ for all } v\in(v_0,v_*).
\eqs
Proceeding as above, we get the existence of discontinuous solution $U$ of the form \eqref{solU} constructed from $V_-$ and $V_+$ lying on $\mathcal{M}_0^-$ and $\mathcal{M}_{v_0}^+$ respectively. Finally, we note that as long as $g(\gamma)\geq v_m$ no discontinuous solutions with two points discontinuity can be constructed. In conclusion, we have obtained the following result. These constructions are illustrated in Figure~\ref{fig:Case2}.

\begin{figure}[t!]
  \centering
 \subfigure[Case $\mathcal{G}_-(g(\gamma))< 0$.]{  \includegraphics[width=.45\textwidth]{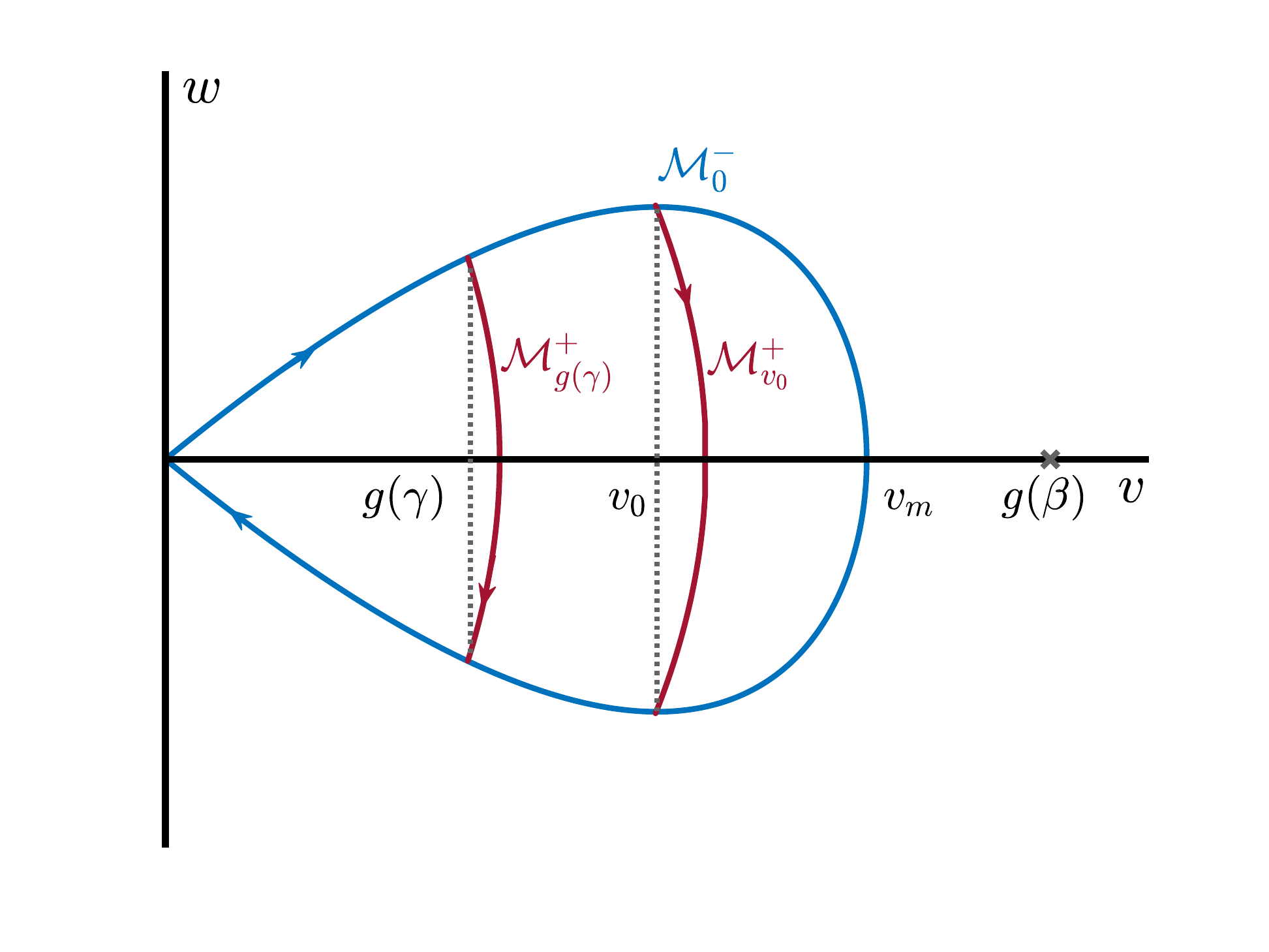} }
  \subfigure[Case $\mathcal{G}_-(g(\gamma))> 0$.]{  \includegraphics[width=.45\textwidth]{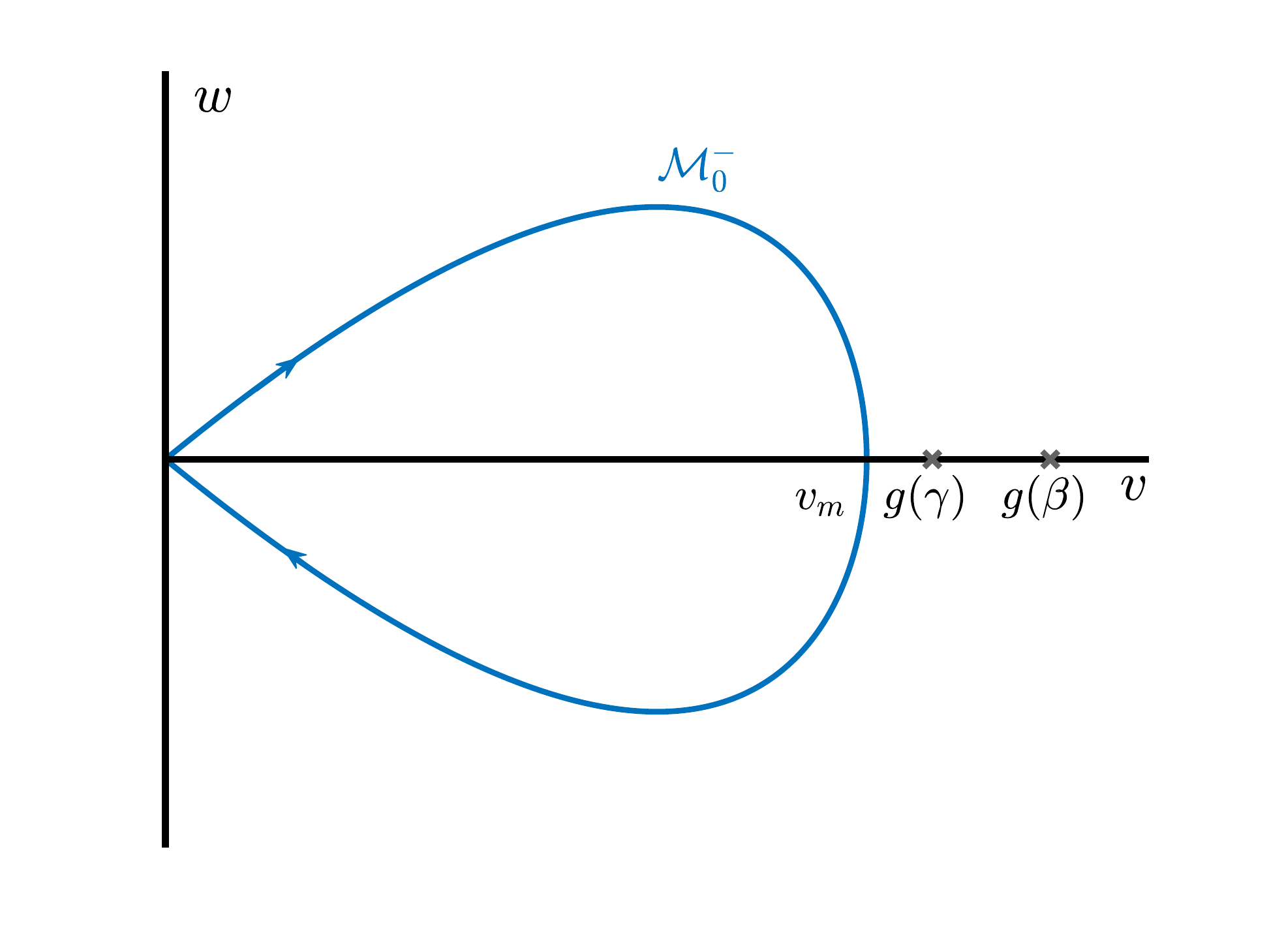} }
  \caption{Illustration of the phase plane analysis and the gluing of trajectories in the case $\beta<\gamma$ \& $\mathcal{G}_-(g(\beta))>0$ with $g(\gamma)>0$ depending on the sign of $\mathcal{G}_-(g(\gamma))$. In blue, we have represented $\mathcal{M}_0^-$ which corresponds to the zero level set of $\mathcal{H}_-$ in the half-plane $v\geq 0$ which forms a homoclinic trajectory to the origin, symmetric in $w$ with extremal value at $v=v_m$. On the left panel (a), when $\mathcal{G}_-(g(\gamma))< 0$,  we have shown in red two trajectories joining the upper part of $\mathcal{M}_0^-$ to its lower part illustrating the fact that there is a one-parameter family of parametrized curves $\mathcal{M}_{v_0}^+$ for each $v_0\in[g(\gamma),v_m]$ which are bounded to the left by $\mathcal{M}_{g(\gamma)}^+$ and to the right by $(v_m,0)$. On the right panel (b), when $\mathcal{G}_-(g(\gamma))> 0$, we have the strict ordering $v_m< g(\gamma) < g(\beta)$ and there is no other possible trajectory than $\mathcal{M}_0^-$.}
  \label{fig:Case2}
\end{figure}

\begin{lem}\label{lem44}
Assume that $g$ satisfies Hypothesis~\ref{hypg}. Assume that $\beta<\gamma$ and $\mathcal{G}_-(g(\beta))>0$. Then there always exists a smooth, even, positive function $U_s:\R \to [0,g_-^{-1}(v_m)]$ solution to \eqref{GS} which is given by $U_s(x) = g_-^{-1}(V_s(x))$ for all $x\in\R$ where $V_s$ is the unique positive homoclinic orbit to the origin of $V-V''=g_{-}^{-1}(V)$. Furthermore, we have:
\begin{itemize}
\item[(i)] if $g(\gamma)>0$ and $\mathcal{G}_-(g(\gamma))\geq 0$, then there is no other solution to \eqref{GS};
\item[(ii)] if $g(\gamma)\leq0$ or $g(\gamma)>0$ and $\mathcal{G}_-(g(\gamma))< 0$, then there is a one parameter family of discontinuous ground state solutions $U$ of equation \eqref{GS} which are even, positive, with precisely two points of jump discontinuity and smooth otherwise, and there is no other solution.
\end{itemize}
\end{lem}

\paragraph{Case $\beta<\gamma$ \& $\mathcal{G}_-(g(\beta))=0$.} If we suppose that $\mathcal{G}_-(g(\beta))=0$, then we necessarily have that
\bqs
\mathcal{G}_-(V) < 0, \text{ for all } V\in(0,g(\beta)).
\eqs
As a consequence, we can proceed exactly as in the case where $\mathcal{G}_-(g(\beta))<0$ to prove the existence of a one parameter family of discontinuous ground state solutions $U$ of the form \eqref{solU} which are even, positive, with precisely two points of jump discontinuity and smooth otherwise. The main difference is that now the zero level set of $\mathcal{H}_-$ in the half-plane $v\geq0$, which is parametrized by 
\bqs
\mathcal{M}_0^-:=\left\{ (v,w)\in\R^2 ~|~ w = \pm \sqrt{-2\mathcal{G}_-(v)}, \quad 0 \leq v \leq g(\beta) \right\},
\eqs
is a closed curve since $\mathcal{G}_-(g(\beta))=0$. This gives a homoclinic trajectory to the origin and thus produces a solution $V_s:\R\to[0,g(\beta)]$ to $V-V''=g_{-}^{-1}(V)$ which satisfies $V_s(\pm\infty)=0$. At the origin we have $V_s(0)=g(\beta)$ such that $U_s(x)=g_{-}^{-1}(V_s(x))$ is only continuous at $x=0$ and smooth elsewhere. This can be summarized as follows.

\begin{figure}[t!]
  \centering
  \includegraphics[width=.4\textwidth]{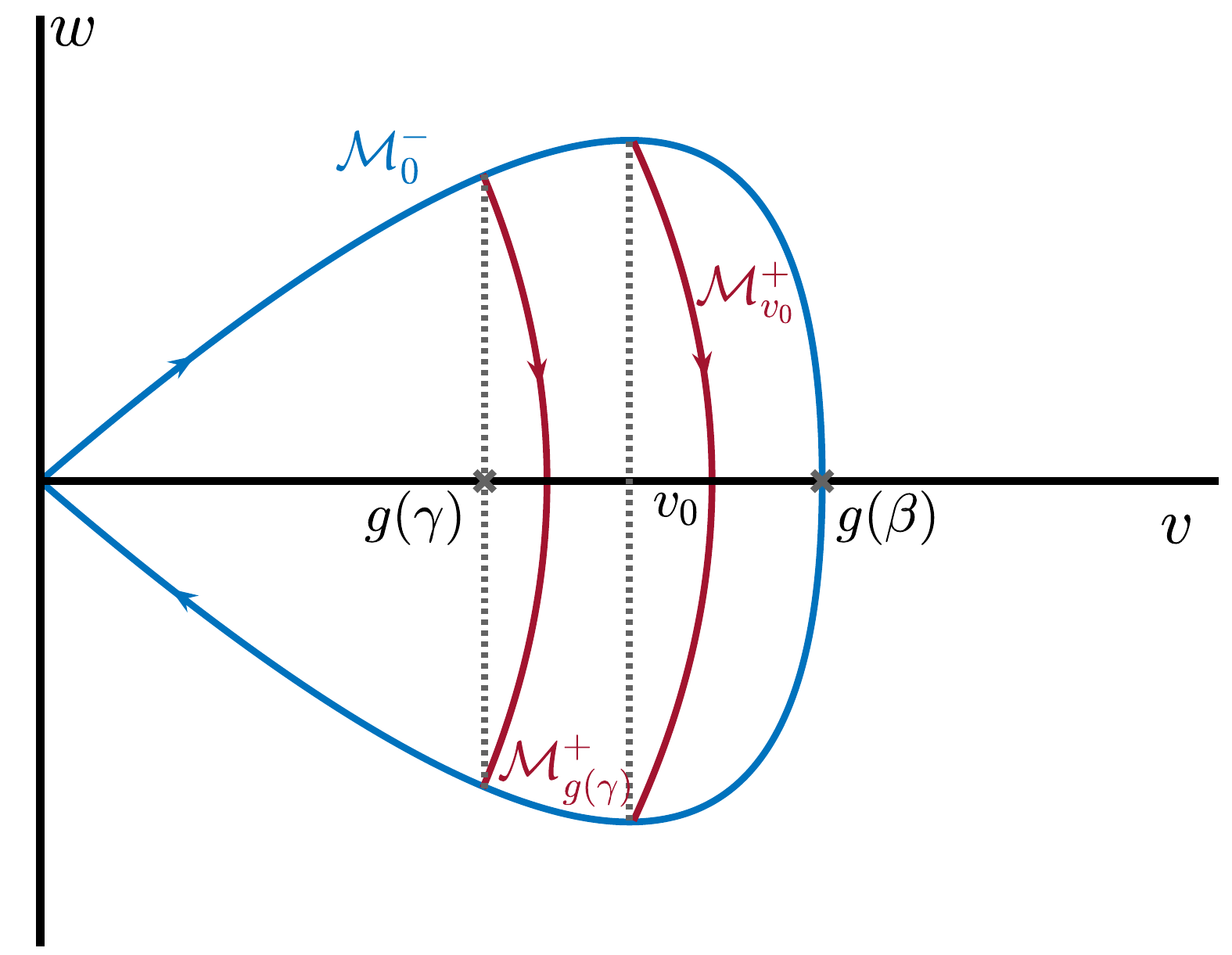}
  \caption{Illustration of the phase plane analysis and the gluing of trajectories in the case $\beta<\gamma$ \& $\mathcal{G}_-(g(\beta))=0$ with $g(\gamma)>0$. In blue, we have represented $\mathcal{M}_0^-$ which corresponds to the zero level set of $\mathcal{H}_-$ in the half-plane $v\geq 0$ which forms a homoclinic trajectory to the origin, symmetric in $w$ with extremal value at $v=g(\beta)$. We have shown in red two trajectories joining the upper part of $\mathcal{M}_0^-$ to its lower part illustrating the fact that there is a one-parameter family of parametrized curves $\mathcal{M}_{v_0}^+$ for each $v_0\in[g(\gamma),g(\beta)]$ which are bounded to the left by $\mathcal{M}_{g(\gamma)}^+$ and to the right by $(g(\beta),0)$.}
  \label{fig:Case3}
\end{figure}

\begin{lem}\label{lem45}
Assume that $g$ satisfies Hypothesis~\ref{hypg}. Assume that $\beta<\gamma$ and $\mathcal{G}_-(g(\beta))=0$. Then, there exists a one parameter family of discontinuous ground state solutions of equation \eqref{GS} given by \eqref{solU} which are even, positive, with precisely two points of jump discontinuity and smooth otherwise together with an even, positive solution which is smooth on $(-\infty,0)\cup (0,+\infty)$ and continuous at the origin. Furthermore, there is no other solution to \eqref{GS}.
\end{lem}

When $\beta=\gamma$, we can reproduce the analysis conducted above and we obtain.

\begin{lem}\label{lembetaeqgam}
Assume that $g$ satisfies Hypothesis~\ref{hypg} and that $\beta=\gamma$. Then,
\begin{itemize}
\item[(i)] if $\mathcal{G}_-(g(\beta))>0$, then there exists a unique smooth, even, positive ground state;
 \item[(ii)] if $\mathcal{G}_-(g(\beta))=0$, then there exists a unique even, positive ground state which is smooth on $(-\infty,0)\cup (0,+\infty)$ and continuous at the origin;
 \item[(iii)] if $\mathcal{G}_-(g(\beta))<0$, then there exists a unique even, positive ground state with precisely two points of jump discontinuity and smooth otherwise.
\end{itemize}
\end{lem}

\begin{rmk}
Note that our Lemma~\ref{lem43},~\ref{lem44},~\ref{lem45} and~\ref{lembetaeqgam} complement and precise \cite[Proposition 3.2 and 3.3]{CR99} by providing an exact characterization for each case and an analytical formula for the point of discontinuity of discontinuous ground state solutions.
\end{rmk}

\subsection{Further properties of $x_0$}

In this subsection, we gather some results regarding $x_0$, the point of jump discontinuity of ground state solutions, when it exists.

\begin{lem}\label{lemx0}
Assume that $g$ satisfies Hypothesis~\ref{hypg} and that $\beta<\gamma$. Assume that  there exists a one parameter family of ground state solutions with two jump discontinuities at $\pm x_0$, indexed by $v_0\in\mathcal{D}$ for some bounded domain $\mathcal{D}\subset(0,g(\beta)]$. Then we have:
\begin{itemize}
\item[(i)] If $g(\gamma)<0$, then $x_0$ defined in \eqref{defx0} satisifies
\bqs
\underset{v_0\rightarrow0^+}{\lim}x_0 = 0.
\eqs
\item[(ii)] If $g(\gamma)>0$, then there exists $C_\gamma>0$ such that $x_0$ defined in \eqref{defx0} satisifies
\bqs
\underset{v_0\rightarrow g(\gamma)^+}{\lim}x_0 = C_\gamma.
\eqs
\item[(iii)] If $\mathcal{G}_-(g(\beta))< 0$ and $-\mathcal{G}_-(g(\beta)) \geq \mathcal{G}_+(1)-\mathcal{G}_+(g(\beta))$, then $x_0$ defined in \eqref{defx0} satisifies
\bqs
\underset{v_0\rightarrow v_c^-}{\lim}x_0 = +\infty.
\eqs
\item[(iv)] The map $x_0:\mathcal{D}\to \R^+$ is continuous.
\end{itemize}
\end{lem}

\begin{proof}
Recalling the definition of $x_0$, and making explicit the dependence in $v_0$, we have
\bqs
x_0(v_0) = \int_{v_0}^{v_*(v_0)} \frac{\md v}{\sqrt{2\left(\mathcal{G}_+(v_0)-\mathcal{G}_-(v_0)-\mathcal{G}_+(v)\right)}}, \quad v_0\in\mathcal{D},
\eqs
where $v_*(v_0)$ is uniquely defined as $\mathcal{G}_+(v_*(v_0))=\mathcal{G}_+(v_0)-\mathcal{G}_-(v_0)$. First, we remark that the function $v\mapsto \mathcal{G}_+(v)-\mathcal{G}_+(v_0)$ is smooth, strictly increasing on $[v_0,v_*]$ and positive. Following \cite[Chapter 3]{schaaf}, we introduce the change of variable $\Psi(t,v_0)$ such that
\bqs
\mathcal{G}_+(\Psi(t,v_0))-\mathcal{G}_+(v_0)=\frac{t^2}{2}, 
\eqs
together with $v_0=\Psi(0,v_0)$ and $v_*(v_0)=\Psi(t_*(v_0), v_0)$ with $t_*(v_0)=\sqrt{-2\mathcal{G}_-(v_0)}$. The map $t\mapsto \Psi(t,v_0)$ is thus positive increasing on $[0,t_*(v_0)]$ and smooth. Using this change of variable, we obtain that
\bqs
x_0(v_0)=\int_0^{\sqrt{-2\mathcal{G}_-(v_0)}} \frac{\partial_t \Psi(t,v_0)}{\sqrt{-2\mathcal{G}_-(v_0)-t^2}}\md t = \int_0^{\frac{\pi}{2}} \partial_t \Psi\left(\sqrt{-2\mathcal{G}_-(v_0)}\cos(\theta),v_0\right) \md \theta.
\eqs
Note that $\partial_t \Psi(t,v_0)$ satisfies 
\bqs
\partial_t \Psi(t,v_0) \left(-\Psi(t,v_0)+g^{-1}_+(\Psi(t,v_0)) \right) = t, \quad t\in [0,\sqrt{-2\mathcal{G}_-(v_0)}].
\eqs
\begin{itemize}
\item[(i)] Since $g(\gamma)<0$, we have that $\gamma<g^{-1}_+(0)<1$, and there exist some constant $C_0>1$ (independent of $v_0$) and $\epsilon_0>0$ such that 
\bqs
0<\frac{1}{C_0}\leq \frac{1}{-\Psi(t,v_0)+g^{-1}_+(\Psi(t,v_0))} \leq C_0, \quad t\in [0,\sqrt{-2\mathcal{G}_-(v_0)}],
\eqs
for all $v_0\in[0,\epsilon_0]$. As a consequence, we have
\bqs
0<x_0(v_0)  \leq C_0 \sqrt{-2\mathcal{G}_-(v_0)} \rightarrow0, \text{ as } v_0\rightarrow0.
\eqs
\item[(ii)] Since $0<g(\gamma)<1$, we have that there exist some constant $C_0>1$ (independent of $v_0$) and $\epsilon_0>0$ such that 
\bqs
0<\frac{1}{C_0}\leq \frac{1}{-\Psi(t,v_0)+g^{-1}_+(\Psi(t,v_0))} \leq C_0, \quad t\in [0,\sqrt{-2\mathcal{G}_-(v_0)}],
\eqs
for all $v_0\in(g(\gamma),g(\gamma)+\epsilon_0]$. As a consequence, we can pass into the limit in the integral and obtain
\bqs
\underset{v_0\rightarrow g(\gamma)^+}{\lim}x_0(v_0)=\int_0^{\frac{\pi}{2}} \partial_t \Psi\left(\sqrt{-2\mathcal{G}_-(g(\gamma))}\cos(\theta),g(\gamma)\right) \md \theta := C_\gamma.
\eqs
We readily obtain that
\bqs
C_\gamma \geq \frac{\sqrt{-2\mathcal{G}_-(g(\gamma))}}{C_0}>0,
\eqs
since we are in the parameter regime where $\mathcal{G}_-(g(\gamma))<0$.
\item[(iii)] We recall that $v_c\in (0,g(\beta)]\cap[g(\gamma),1)$ is uniquely defined via $\mathcal{G}_+(v_c)-\mathcal{G}_-(v_c) = \mathcal{G}_+(1)$ and that $v_*(v_0)\rightarrow 1$ as $v_0\rightarrow v_c^-$. We also have that $\mathcal{G}_+'(1)=0$ and $\mathcal{G}_+''(1)<0$ so that
\bqs
\underset{v\to 1^-}{\lim}\frac{1-v}{\sqrt{2(\mathcal{G}_+(1)-\mathcal{G}_+(v))}} = \frac{1}{\sqrt{-\mathcal{G}_+''(1)}}\neq 0.
\eqs
As a consequence, we have that
\bqs
x_0(v_0) = \int_{v_0}^{v_*(v_0)} \frac{\md v}{\sqrt{2\left(\mathcal{G}_+(v_*(v_0))-\mathcal{G}_+(v)\right)}} \longrightarrow +\infty, \text{ as }v_0\rightarrow v_c^-.
\eqs
\item[(iv)] By the implicit function theorem, $\Psi$ is continuous in its second argument. Furthermore since for each $v_0\in \mathcal{D}$ and $t\in [0,\sqrt{-2\mathcal{G}_-(v_0)}]$ one has $-\Psi(t,v_0)+g^{-1}_+(\Psi(t,v_0))\neq 0$, we get that $v_0\mapsto \partial_t \Psi\left(\sqrt{-2\mathcal{G}_-(v_0)}\cos(\theta),v_0\right)$ is continuous uniformly in $\theta\in[0,\pi/2]$, which gives the result.
\end{itemize}
\end{proof}
%

\subsection{Application to the cubic nonlinearity}\label{subsecapplcubic}

In the specific case of the cubic nonlinearity $f_a(u)=u(1-u)(u-a)$ with $a\in(0,1/2)$, we can get explicit formula and obtain a complete picture which is summarized in Figure~\ref{fig:DiagGS}. First, we readily note that $\beta$ and $\gamma$ from Hypothesis~\ref{hypg} are well defined if and only if $\frac{1-a+a^2}{3}\geq d$, since in that case they are given explicitly as
\bqs
\beta = \frac{1+a-\sqrt{(1+a)^2-3(a+d)}}{3}, \quad \gamma = \frac{1+a+\sqrt{(1+a)^2-3(a+d)}}{3}.
\eqs
We have that $\beta=\gamma$ precisely when $d=\frac{1-a+a^2}{3}$, which corresponds to the blue curve in Figure~\ref{fig:DiagGS}. Furthermore,  $g'>0$ on $[0,1]$ whenever $d>\frac{1-a+a^2}{3}$ which corresponds to the region $\textbf{I}$. Below the curve $\beta=\gamma$, we distinguish three different regions which are delimited by the curves $\mathcal{G}_-(g(\beta))=0$ and $\mathcal{G}_-(g(\gamma))=0$ which are respectively the green and red curves in Figure~\ref{fig:DiagGS}. We also note that the functions $\mathcal{G}_\pm$ can be expressed as
\bqs
\mathcal{G}_\pm(v) = -\frac{v^2}{2}+\frac{3}{4d}g_\pm^{-1}(v)^4-\frac{2}{3d}(1+a)g_\pm^{-1}(v)^3+\frac{1}{2}\left(1+\frac{a}{d}\right)g_\pm^{-1}(v)^2.
\eqs
Region {\bf II} corresponds to the situation where $\beta<\gamma$ and $\mathcal{G}_-(g(\beta))<0$ where Lemma~\ref{lem43} applies. Region {\bf III} corresponds to the situation where $\beta<\gamma$ and $\mathcal{G}_-(g(\beta))>0$ with $\mathcal{G}_-(g(\beta))<0$ which is the second case (ii) of  Lemma~\ref{lem44} applies. The last region {\bf IV} where $\beta<\gamma$ and $\mathcal{G}_-(g(\beta))>0$ with $\mathcal{G}_-(g(\beta))>0$ corresponds to the first case (i) of  Lemma~\ref{lem44}. Let us finally note that Lemma~\ref{lembetaeqgam} holds on the curve $\beta=\gamma$ and the point where $\mathcal{G}_-(g(\beta))=0$ separates the different cases. All curves interest at unique point marked by the orange cross. Within region {\bf II}, the boundary of the pinning region (pink dashed curve), characterized by \eqref{boundaryPR}, rewrites as \cite{AFSS}:
\bqs
d=d_{\mathrm{pin}}(a)=\frac{1}{3}\left( 1-a+a^2-\sqrt{1-2a}\right), \quad 0<a<1/2.
\eqs

\begin{figure}[t!]
  \centering
  \includegraphics[width=.45\textwidth]{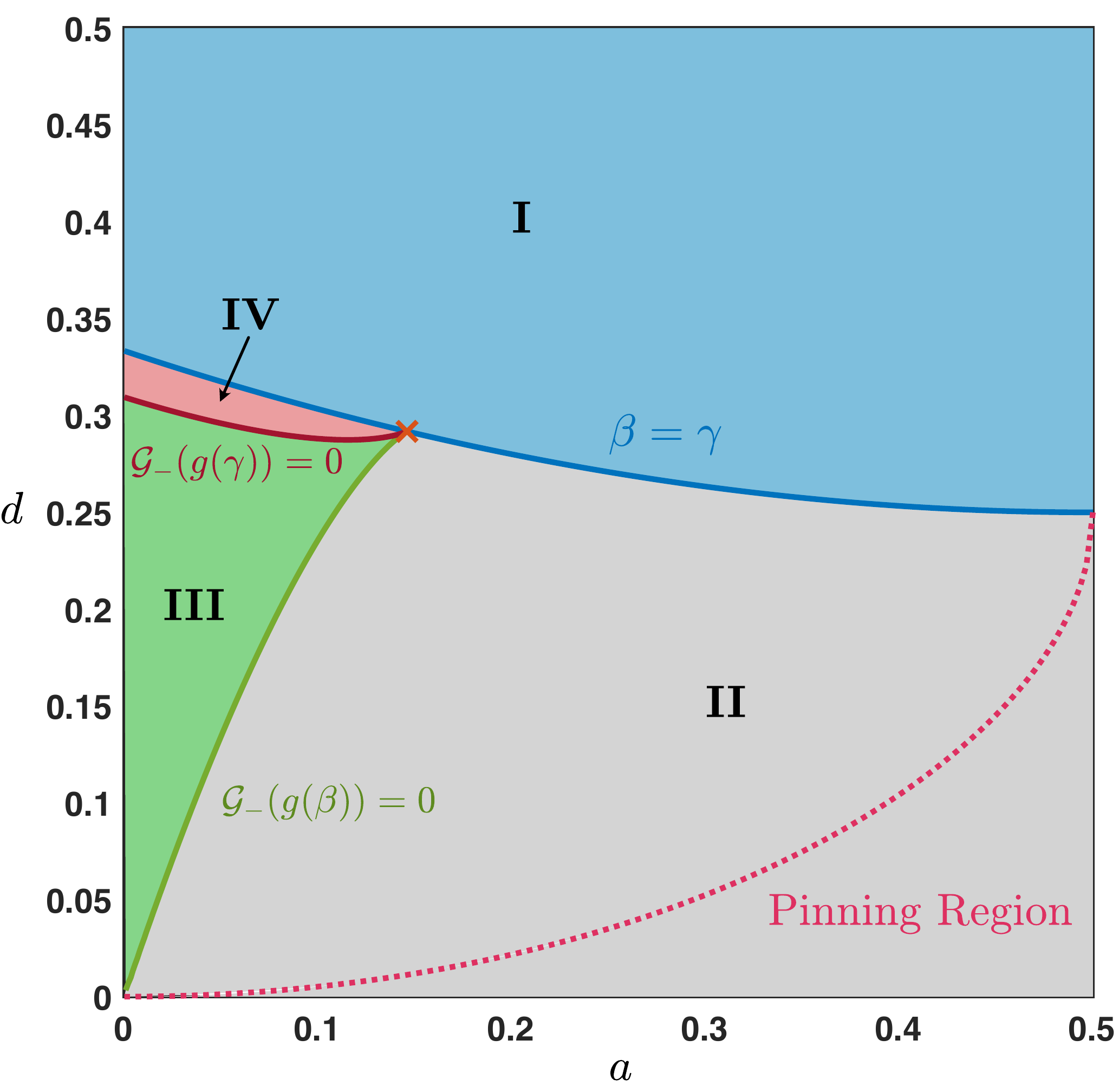}
  \caption{Regions of application of our lemmas in the case of the cubic nonlinearity $f_a(u)=u(1-u)(u-a)$ with $a\in(0,1/2)$ and $d>0$. The boundary between each regions are given by the blue curve, parametrized by $\beta=\gamma$, the dark red curve, parametrized by $\mathcal{G}_-(g(\gamma))=0$, and the green curve, parametrized by $\mathcal{G}_-(g(\beta))=0$. Note that they all intersect at a unique point marked by an orange cross. Lemma~\ref{lem41} applies in region {\bf I} (blue), while Lemma~\ref{lem43} holds true in region {\bf II} (grey). Cases (i) and (ii) of Lemma~\ref{lem44} correspond respectively to regions {\bf IV} (red) and {\bf III} (green). We also represented the boundary of the pinning region (pink dashed curve) given by $d=d_{\mathrm{pin}}(a)$.}
  \label{fig:DiagGS}
\end{figure}

\section{Regime of small diffusion}\label{secsmall}

In section, we investigate the regime of small diffusion and prove several results regarding the non extinction, non propagation and stagnation of the solutions of the Cauchy problem~\eqref{edp}. Our technique of proof relies on comparison principle and the construction of appropriate sub and super-solutions using discontinuous ground state solutions found in the previous section.

\subsection{Non extinction for small diffusion}\label{secnonext}

In this subsection, we prove that there is always persistence of the solution in the special case where $\K(x)=e^{-|x|}/2$ and when $g$ verifies Hypothesis~\ref{hypg}. First, we introduce some notations. By definition of $\gamma \in(0,1)$ one has $g'(\gamma)=0$ which reads equivalently as $d=f'(\gamma)>0$. On the other hand, since $f'(1)<0$, there exists $u_+\in(\gamma,1)$ such that $f'(u_+)=0$ and for all $u\in(u_+,1]$ one has $f'(u)<0$. Similarly, there exists $u_-\in(0,\beta)$ such that $f'(u_-)=0$ and for all $u\in[0,u_-)$ one has $f'(u)<0$.

\begin{prop}[Non extinction]\label{prop51}
Let $\K(x)=e^{-|x|}/2$ and assume that $g$ satisfies Hypothesis~\ref{hypg} and that $\beta<\gamma$ with $g(u_+)<0$. Then, there exists $\widetilde{\ell}_0>0$ such that for all $\ell \in(0,\widetilde{\ell}_0)$ the solution $u_\ell$ of \eqref{edp} satisfies
\bqs
U (x) \leq \underset{t\rightarrow+\infty}{\liminf} \, u_\ell(t,x)\ \text{ uniformly in } x\in\R\backslash\{\pm\ell\},
\eqs
where $U$ is a discontinuous ground state of \eqref{GS} with discontinuity points precisely at $\pm\ell$.
\end{prop} 

\begin{proof}
We first remark that since $u_+\in(\gamma,1)$ is such that $g(u_+)<0$, one necessarily has that $g(\gamma)<0$. As a consequence, using Lemma~\ref{lem43} and Lemma~\ref{lem44}, there exists a one parameter family of discontinuous ground state solutions $U$ of \eqref{GS} indexed by $v_0\in \mathcal{D}$  for some bounded domain $\mathcal{D}\subset (0,g(\beta)]$ with two jump discontinuities at $\pm \, x_0(v_0)$ which verifies, by Lemma~\ref{lemx0}, 
\bqs
\underset{v_0\rightarrow0^+}{\lim}x_0(v_0) = 0\,.
\eqs
Since $g(1)=1$ and $g(u_+)<0$, one gets the existence of $u_\star\in(u_+,1)$ such that $g(u_\star)=0$ by continuity of $g$. As a consequence, using the fact that $f'(0)<0$ and the monotonicity properties of $f$, there exist some small $\eta \in (0,\min((u_\star-u_+)/2,u_-/2))$ and some constant $c>0$ such that $f'(u)<-c<0$ for each $u\in(-\eta,u_-/2)\cup(u_\star-\eta,1)$. We refer to Figure~\ref{fig:figuregext} for an illustration of the relative position between all the constants.

We then define
\bqs
0<\widetilde{\ell}_0:=\underset{v_0\in(0,g(\eta/2)) \, \cap \, \mathcal{D}}{\sup}\, x_0(v_0)\,.
\eqs
For each $\ell\in (0,\widetilde{\ell}_0)$ there exist some $v_0 \in(0,g(\eta/2))\, \cap\, \mathcal{D}$ such that $\ell=x_0(v_0)$ and a discontinuous ground state $U$ of \eqref{GS} with discontinuity points precisely at $\pm\ell$. We are going to prove that 
\bqs
\underline{u}(t,x)=U(x)-\delta e^{-\alpha t},
\eqs
is a sub-solution for the Cauchy problem \eqref{edp} for some well-chosen constants $\delta>0$ and $\alpha>0$. Recalling the definition of $\mathscr{N}$ in \eqref{eqNcal}, we first compute
\begin{align*}
 \mathscr{N}( \underline{u}(t,x))&=\delta \alpha e^{-\alpha t}+ f(U(x))-f(U(x)-\delta e^{-\alpha t})\\
 &=\delta  e^{-\alpha t}\left(\alpha+ f'(\xi(t,x))\right),
 \end{align*}
 for some $\xi(t,x)\in(U(x)-\delta e^{-\alpha t},U(x))$. We now want to ensure that $f'(\xi(t,x))<0$ for each $t>0$ and $x\in\R$. We let $\delta \in (\eta/2,\eta)$. For $|x|>\ell$, we have that $\xi(t,x)\in(U(x)-\delta e^{-\alpha t},U(x)) \subset(-\delta,g_-^{-1}(v_0))) \subset(-\eta,u_-/2)$, since for each $v_0\in(0,g(\eta/2))$ one has $0<g_-^{-1}(v_0)<\eta/2<u_-/4<u_-/2$. As a consequence,  we get that $f'(\xi(t,x))<-c<0$ for $|x|>\ell$. On the other hand, for $|x| \leq \ell$, we have $\xi(t,x)\in(U(x)-\delta e^{-\alpha t},U(x)) \subset(g_+^{-1}(v_0)-\delta,1)$. But for each $v_0\in(0,g(\eta/2))$, one has that $g_+^{-1}(v_0)>u_\star$ and so $g_+^{-1}(v_0)-\delta >u_\star-\eta$ since $\delta<\eta$. As a consequence, we get that  $f'(\xi(t,x))<-c<0$ for $|x| \leq \ell$. We then set $\alpha=c/2$ and we have proved that
 \bqs
  \mathscr{N}( \underline{u}(t,x)) < 0,
 \eqs
 for each $t>0$ and $x\in\R$. It only remains to check that $U(x)-\delta \leq \mathds{1}_{[-\ell,\ell]}(x)$ for all $x\in\R$. The inequality is trivially satisfied for any $|x|\leq \ell$. For $|x|>\ell$, we have that $U(x)-\delta \leq g_-^{-1}(v_0)-\delta < \eta/2 -\delta<0$ since $\delta \in (\eta/2,\eta)$. Using the comparison principle, we get that
 \bqs
 U(x)-\delta e^{-\alpha t} \leq u_\ell(t,x),
 \eqs
 for each $t>0$ and $x\in\R$, and the conclusion of the proposition follows.
\end{proof}

\begin{figure}[t!]
  \centering
  \includegraphics[width=.5\textwidth]{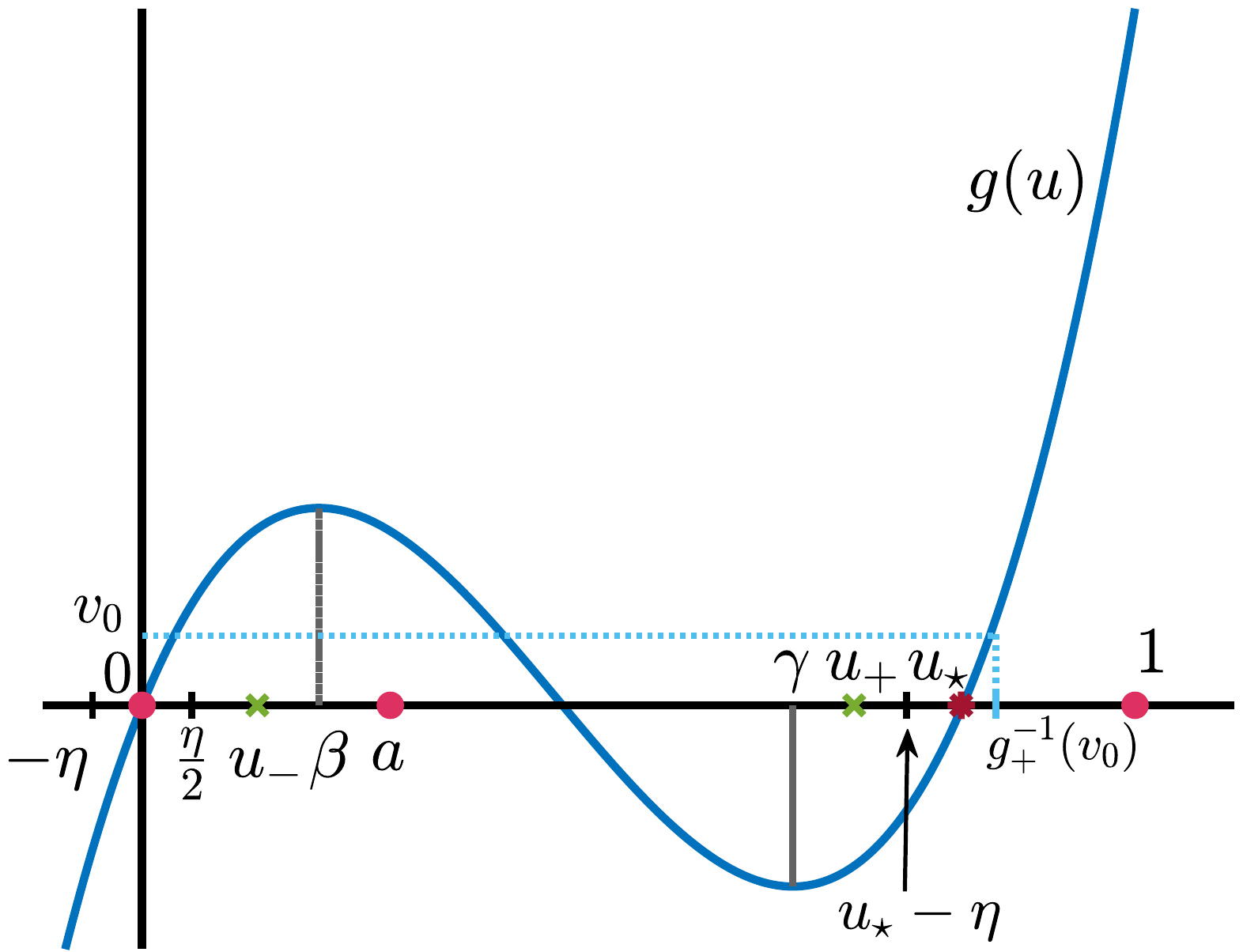}
  \caption{Illustration of a function $g$ satisfying Hypothesis~\ref{hypg} with $\beta<\gamma$ and $g(u_+)<0$ such that Proposition~\ref{prop51} applies. We refer to the proof of Proposition~\ref{prop51} for a precise definition of $\eta>0$. }
  \label{fig:figuregext}
\end{figure}

The above proposition shows that if $\ell>0$ is small enough, then the solution $u_\ell$ of the Cauchy problem is always bounded below by a ground state solution with discontinuity points which are exactly at $\pm \ell$. In fact, if one relaxes this latter condition, we can actually prove that for all $\ell>0$, the solution $u_\ell$ of the Cauchy problem is always bounded below by a ground state solution with discontinuity points which are now at $\pm x_0$ with $x_0 \leq \ell$.

\begin{prop}[Non extinction]\label{prop52}
Let $\K(x)=e^{-|x|}/2$ and assume that $g$ satisfies Hypothesis~\ref{hypg} and that $\beta<\gamma$ with $g(u_+)<0$. Then, for all $\ell>0$ the solution $u_\ell$ of \eqref{edp} satisfies
\bqs
U (x) \leq \underset{t\rightarrow+\infty}{\liminf} \, u_\ell(t,x)\ \text{ uniformly in } x\in\R\backslash\{\pm x_0\},
\eqs
where $U$ is a discontinuous ground state of \eqref{GS} with discontinuity points at $\pm x_0$ with $0<x_0\leq \ell$.
\end{prop} 

\begin{proof}
From our assumption, we know that there exist some small $\eta \in (0,\min((u_\star-u_+)/2,u_-/2))$ and some constant $c>0$ such that $f'(u) \leq -c<0$ for each $u\in[-\eta,u_-/2]\cup[u_\star-\eta,1]$. Finally, there exists some $\epsilon_0>0$ such that for each $v_0\in(0,\epsilon_0)$, one has that $g_-^{-1}(v_0)<u_-/2$. Now, for each $v_0\in(0,\epsilon_0)$, we denote $U_{v_0}$ the discontinuous ground state solution with discontinuity points precisely at $\pm x_0(v_0)$, and we let
\bqs
\underline{u}(t,x)=U_{v_0}(x)-\eta e^{-\frac{ct}{2} }.
\eqs
For $|x|>x_0(v_0)$, we have that $(U_{v_0}(x)-\eta e^{-\frac{ct}{2}},U_{v_0}(x)) \subset [-\eta,g_-^{-1}(v_0))] \subset [-\eta,u_-/2]$, while for $|x| \leq x_0(v_0)$, we have that $(U_{v_0}(x)-\eta e^{-\frac{ct}{2}},U_{v_0}(x)) \subset [g_+^{-1}(v_0)-\eta,1] \subset [u_\star-\eta,1]$. As a consequence, for any $t>0$ and $x\in\R$, one gets that $\mathscr{N}( \underline{u}(t,x)) < 0$.

We now let $\ell>0$. Since both $x_0(v_0)\rightarrow0$ as as $v_0\rightarrow 0$ and  $g_-^{-1}(v_0)\rightarrow 0$ as $v_0\rightarrow 0$, one can find some $\epsilon\in(0,\epsilon_0)$ such that for all $v_0\in(0,\epsilon)$ one has $g_-^{-1}(v_0)<\eta/2$ and $0<x_0(v_0)\leq \ell$. As a consequence, we have that $U_{v_0}(x)-\eta \leq \mathds{1}_{[-\ell,\ell]}(x)$ for all $x\in\R$ for each $v_0\in(0,\epsilon)$. Then, the comparison principle gives that 
\bqs
 U_{v_0}(x)-\eta e^{-\frac{ct}{2}} \leq u_\ell(t,x),
 \eqs
 for each $t>0$ and $x\in\R$, and the proof is complete.

\end{proof}

\subsection{Non propagation in the pinning region}\label{secnonprop}

In this subsection, we prove that there is non propagation of the solution in the special case where $\K(x)=e^{-|x|}/2$ and when $g$ verifies Hypothesis~\ref{hypg}. 

\begin{prop}[Non propagation]\label{prop53}
Let $\K(x)=e^{-|x|}/2$ and assume that $g$ satisfies Hypothesis~\ref{hypg} and that $\beta<\gamma$ with $\mathcal{G}_-(g(\beta))<0$ and $-\mathcal{G}_-(g(\beta))\leq \mathcal{G}_+(1)-\mathcal{G}_+(g(\beta))$. Let $v_c\in (0,g(\beta)]\cap [g(\gamma),1)$ such that $\mathcal{G}_+(v_c)-\mathcal{G}_-(v_c)=\mathcal{G}_+(1)$ and assume that $g(u_+)<v_c<g(u_-)$. Then, for all $\ell > 0$ the solution $u_\ell$ of \eqref{edp} satisfies
\bqs
 \underset{t\rightarrow+\infty}{\limsup} \, u_\ell(t,x) \leq U(x) \, \text{ uniformly in } x\in\R\backslash\{\pm x_0\},
\eqs
for some discontinuous ground state  $U$ of \eqref{GS} with discontinuity points precisely at $\pm x_0$ with $x_0 \geq \ell$.
\end{prop} 

\begin{proof}
Using Lemma~\ref{lem43}, we know that there exists a one parameter family of discontinuous ground state solutions $U_{v_0}$ of \eqref{GS} with discontinuity points precisely at $\pm x_0(v_0)$ for each $v_0\in(\max(g(\gamma),0),v_c)$. Here, we make explicit the dependence of $U$ with respect to $v_0$.  Furthermore Lemma~\ref{lemx0}(iii) ensures that the corresponding discontinuity points verify
\bqs
\underset{v_0\rightarrow v_c^-}{\lim}x_0(v_0) = +\infty.
\eqs
Since we assume that $v_c<g(u_-)$ and $f'(1)<0$ by hypothesis, there exists some $\delta>0$ such that $g_-^{-1}(v_c)+\delta <u_-$ and $f'(1+\delta)<0$. Now since, $g(u_+)<v_c$, there exists some $\eta>0$ such that $\eta+ u_+<g_+^{-1}(v_c)$ and $\epsilon_\eta>0$ such that for all $v_0\in(v_c-\epsilon_\eta,v_c)$ one has $\eta+ u_+<g_+^{-1}(v_0)$. For each $v_0\in(v_c-\epsilon_\eta,v_c)$, we let 
\bqs
\overline{u}(t,x)=U_{v_0}(x)+\delta e^{-\alpha t},
\eqs
where $U_{v_0}$ is a discontinuous ground state of \eqref{GS} with discontinuity points precisely at $\pm x_0(v_0)$ for some well-chosen $\alpha>0$ and $\delta>0$ defined above. First, we compute
\begin{align*}
 \mathscr{N}( \overline{u}(t,x))&=-\delta \alpha e^{-\alpha t}+ f(U_{v_0}(x))-f(U_{v_0}(x)+\delta e^{-\alpha t})\\
 &=-\delta  e^{-\alpha t}\left(\alpha+ f'(\xi(t,x))\right),
 \end{align*}
 for some $\xi(t,x)\in(U_{v_0}(x),U_{v_0}(x)+\delta e^{-\alpha t})$. For any $|x|\geq x_0(v_0)$, we have $\xi(t,x)\in (0, g_-^{-1}(v_0)+\delta) \subset (0, g_-^{-1}(v_c)+\delta)$ since $v_0<v_c$ and $g_-^{-1}$ is increasing on $[0,\beta)$. As a consequence, one has that 
 \bqs
 f'(\xi(t,x)) \leq \underset{u\in  [0, g_-^{-1}(v_c)+\delta]}{\sup} f'(u):= -m_1<0, \quad |x|\geq x_0(v_0).
 \eqs
 On the other hand, when $|x| < x_0(v_0)$, we have $\xi(t,x)\in (g_+^{-1}(v_0),1+\delta)\subset (\eta+u_+,1+\delta)$ since $v_0\in(v_c-\epsilon_\eta,v_c)$. As a consequence, one has that 
 \bqs
 f'(\xi(t,x)) \leq \underset{u\in  [\eta+u_+, 1+\delta]}{\sup} f'(u) :=- m_2<0, \quad |x|< x_0(v_0).
 \eqs
And we then set $\alpha= \frac{1}{2} \min(m_1,m_2)>0$ such that for all $t>0$ and $x\in\R$, one has
\bqs
 \mathscr{N}( \overline{u}(t,x))>0.
\eqs
To conclude the proof, we let $\ell>0$. Since we have that $U_{v_0}\rightarrow1$ as $v_0\rightarrow v_c^{-}$ on any compact set, there exists $\epsilon \in (0, \epsilon_\eta)$ such that for all $v_0\in(v_c-\epsilon ,v_c)$ one has
\bqs
1\leq U_{v_0}(x)+\delta, \quad x\in[-\ell,\ell],
\eqs
such that for all $x\in\R$ one gets $u_\ell(t=0,x)=\mathds{1}_{[-\ell,\ell]}(x)\leq U_{v_0}(x)+\delta$, and the comparison principle gives that
\bqs
u_\ell(t,x)\leq U_{v_0}(x)+\delta e^{- \alpha t},
\eqs
for each $v_0\in(v_c-\epsilon ,v_c)$. The proof is thus complete.
\end{proof}

\subsection{Towards stagnation for small diffusion}

We can combine Proposition~\ref{prop52} and Proposition~\ref{prop53} to obtain the following corollary which essentially says that for small enough diffusion and parameters within the pinning region, solutions of the Cauchy problem can neither go extinct nor propagate and are bounded above and below by two different ground state solutions.

\begin{cor}[Stagnation]\label{cor}
Let $\K(x)=e^{-|x|}/2$ and assume that $g$ satisfies Hypothesis~\ref{hypg} and that $\beta<\gamma$ with $\mathcal{G}_-(g(\beta))<0$ and $-\mathcal{G}_-(g(\beta))\leq \mathcal{G}_+(1)-\mathcal{G}_+(g(\beta))$. Let $v_c\in (0,g(\beta)]\cap [g(\gamma),1)$ such that $\mathcal{G}_+(v_c)-\mathcal{G}_-(v_c)=\mathcal{G}_+(1)$ and assume that $g(u_+)<0<v_c<g(u_-)$. Then, for all $\ell > 0$ the solution $u_\ell$ of \eqref{edp} satisfies
\bqs
\underline{U}(x) \leq \underset{t\rightarrow+\infty}{\liminf} \, u_\ell(t,x) \leq \underset{t\rightarrow+\infty}{\limsup} \, u_\ell(t,x) \leq \overline{U}(x) \, \text{ uniformly in } x\in\R\backslash\{\pm \underline{x}_0,\pm \overline{x}_0\},
\eqs
for some discontinuous ground states  $\underline{U}$ and $\overline{U}$ of \eqref{GS} with discontinuity points precisely at $\pm \underline{x}_0$ and $\pm \overline{x}_0$ respectively with $0<\underline{x_0} \leq \ell \leq \overline{x_0}$.
\end{cor}

\section{Numerical illustrations}\label{secnum}

In this section, we present our numerical results which illustrate and complement our theoretical findings of the previous sections. We first present the numerical scheme which we used in order to numerically approximate the solutions of the Cauchy problem \eqref{edp} and then present and comment our numerical results. In particular, one of our main goal is to numerically compute the values of $\ell_0^*$ and $\ell_1^*$ with high accuracy for any given nonlinearity $f$, kernel $\K$ and coefficient diffusion $d>0$ satisfying our assumptions. In order to achieve this objective, we decided to use a high order numerical scheme both in space and time which is presented in the first subsection.  Most interestingly, our numerical explorations allowed us to conjecture an analytical formula for the value $\ell_1^*$ in a well identified region of parameters for the special case of the cubic nonlinearity and the exponential kernel function.

\subsection{Numerical method}

In order to numerically compute the solutions of the Cauchy problem \eqref{edp}, we use the splitting method introduced by Descombes \& Schatzman \cite{DS95}, which is formally of order $4$ (see \cite{Desc01} and references therein). This splitting method is obtained by applying a Richardson's extrapolation to the standard Strang's method \cite{Strang}. More precisely, let us first denote by $\mathcal{S}^t$ the flow associated to \eqref{edp}, that is, to a given initial condition $u_0$, the solution to \eqref{edp} at time $t>0$ is given by $\mathcal{S}^tu_0$. Next, we introduce two intermediate Cauchy problems. The first one is the following nonlocal diffusion equation
\bqq
\label{heat}
\left\{
\begin{split}
\partial_t v &= d\left(-v+\K *v\right), \quad t>0 \quad x\in\R, \\
v(t=0,x) & = v_0(x), \quad x\in\R,
\end{split}
\right.
\eqq
whose flow is denoted by $\mathcal{D}^t$ for $t>0$. While the second one is given by
\bqq
\label{fode}
\left\{
\begin{split}
\partial_t w &= f(w), \quad t>0 \quad x\in\R, \\
w(t=0,x) & = w_0(x), \quad x\in\R,
\end{split}
\right.
\eqq
and we denote its flow by $\mathcal{R}^t$ for $t>0$. We denote by $\mathcal{Z}^t:=\mathcal{R}^{t/2}\mathcal{D}^t\mathcal{R}^{t/2}$ for $t>0$ the composition of the flows generated by \eqref{heat} and \eqref{fode}, which precisely corresponds to Strang's method \cite{Strang}. The splitting method introduced in \cite{DS95} and that we have used to produce our numerical simulations is defined as
\bqs
\mathcal{Y}^t:=\frac{4}{3}\mathcal{Z}^{t/2}\mathcal{Z}^{t/2}-\frac{1}{3}\mathcal{Z}^t, \quad t>0.
\eqs
We have used an explicit Runge-Kutta method of order 4 to solve the ODE problem \eqref{fode}, while we have relied on a spectral method to solve the nonlocal diffusion equation \eqref{heat} by noticing that 
\bqs
(\mathcal{D}^tv_0)(x) =  \frac{1}{2\pi}\int_\R \mathrm{e}^{\mathbf{i} \xi x} \mathrm{e}^{d(-1+\widehat{\K}(\xi))t} \widehat{v_0}(\xi)\md \xi, \quad t>0, \quad x \in \R,
\eqs
where $\widehat{\K}$ and $ \widehat{v_0}$ denote the Fourier transform of $\K$ and $v_0$ respectively. For $\K(x)=\mathrm{e}^{-|x|}/2$ we have the exact formula $\widehat{\K}(\xi) = (1+\xi^2)^{-1}$.

Regarding our time and space discretization, we typically used $N=2^{14}$ points for the spectral method and a time step of $\Delta t=0.01$. The spatial domain was set to $[-10\pi,10\pi]$ while the final time of computation was adapted to the situation under investigation but never exceeded $T_f=500$. Our numerical findings are reported in the next section.

\subsection{Results for the cubic nonlinearity and the exponential kernel}

\begin{figure}[t!]
  \centering
  \includegraphics[width=.475\textwidth]{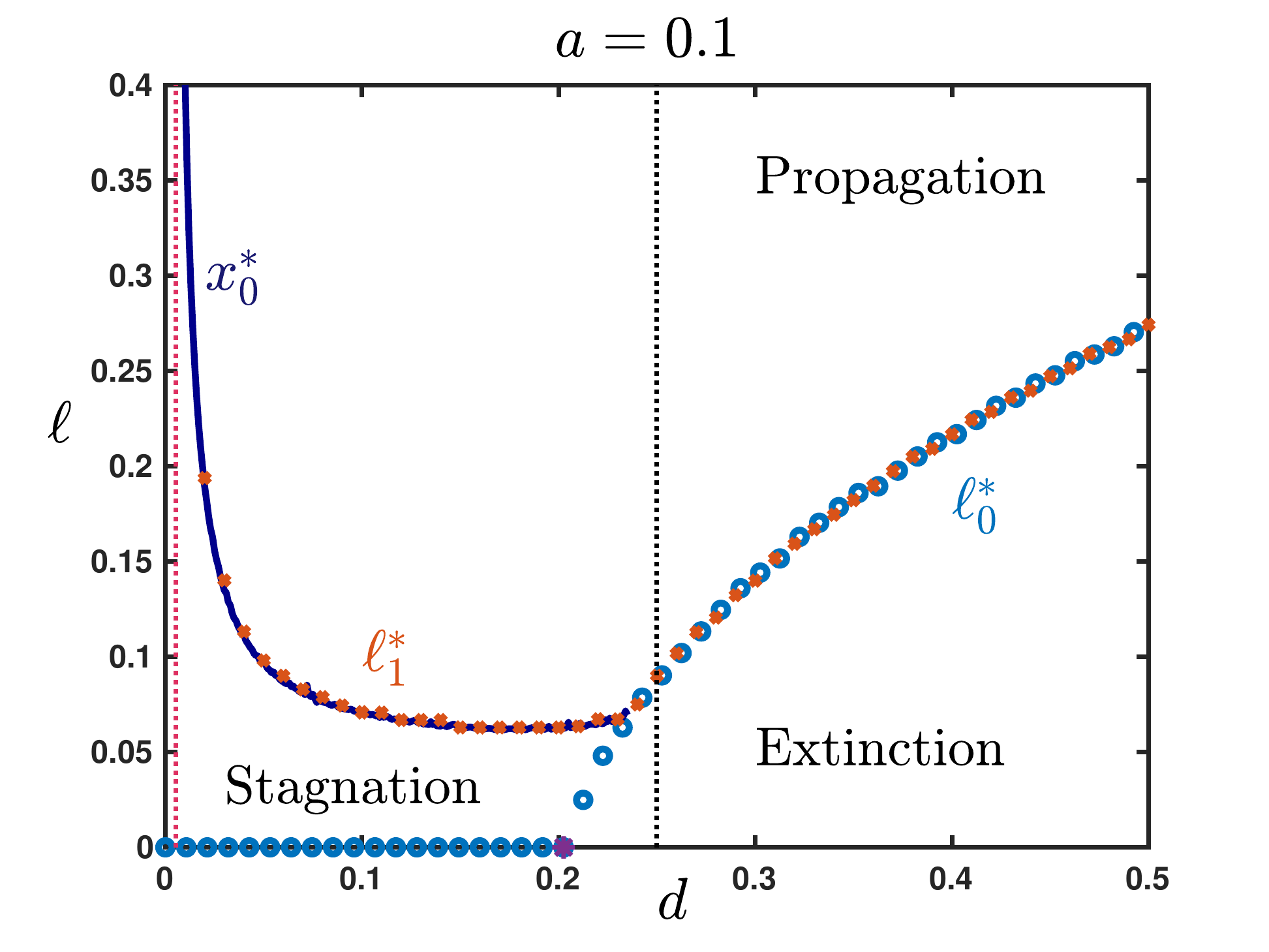}
  \includegraphics[width=.475\textwidth]{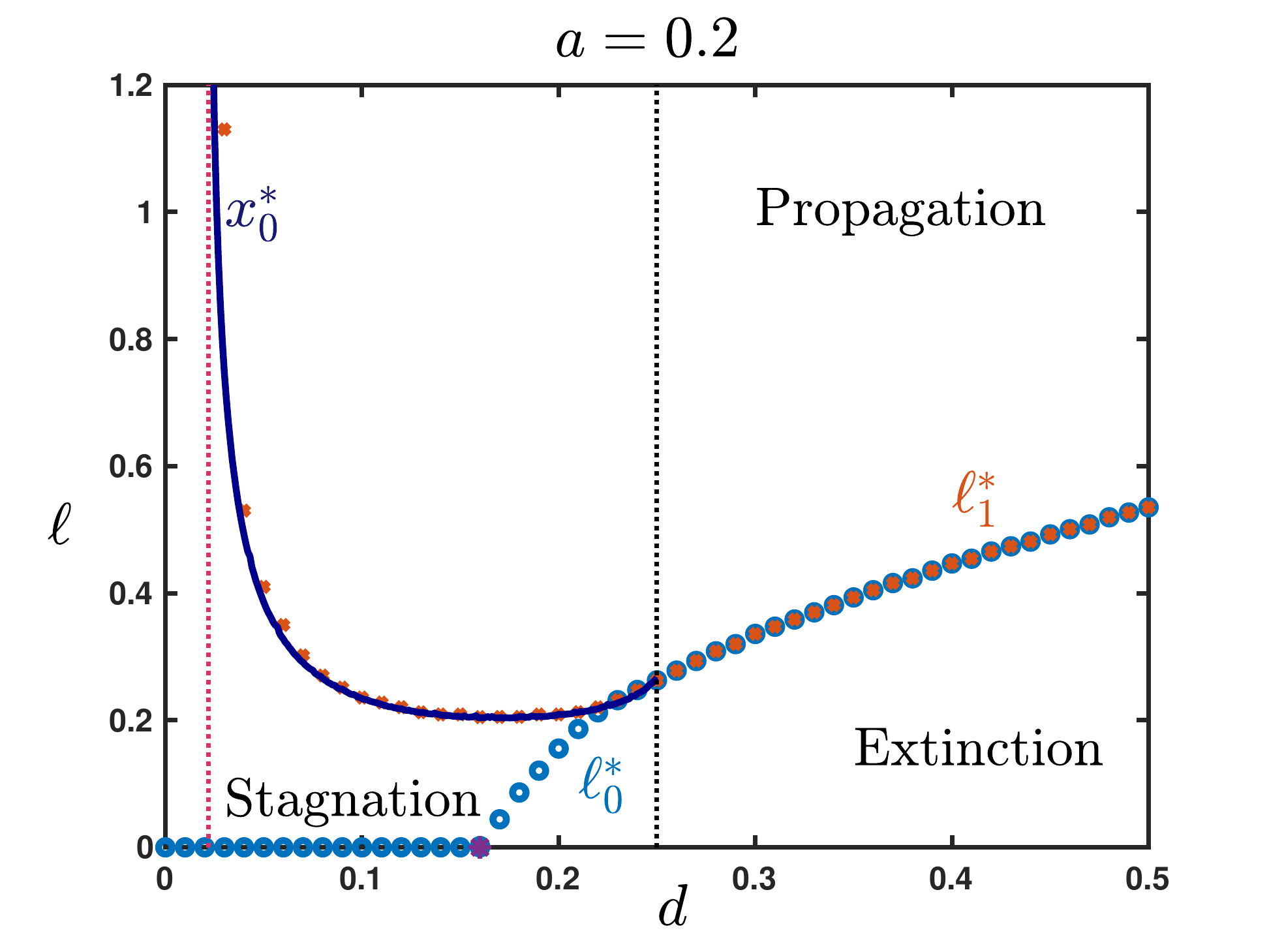}
  \includegraphics[width=.475\textwidth]{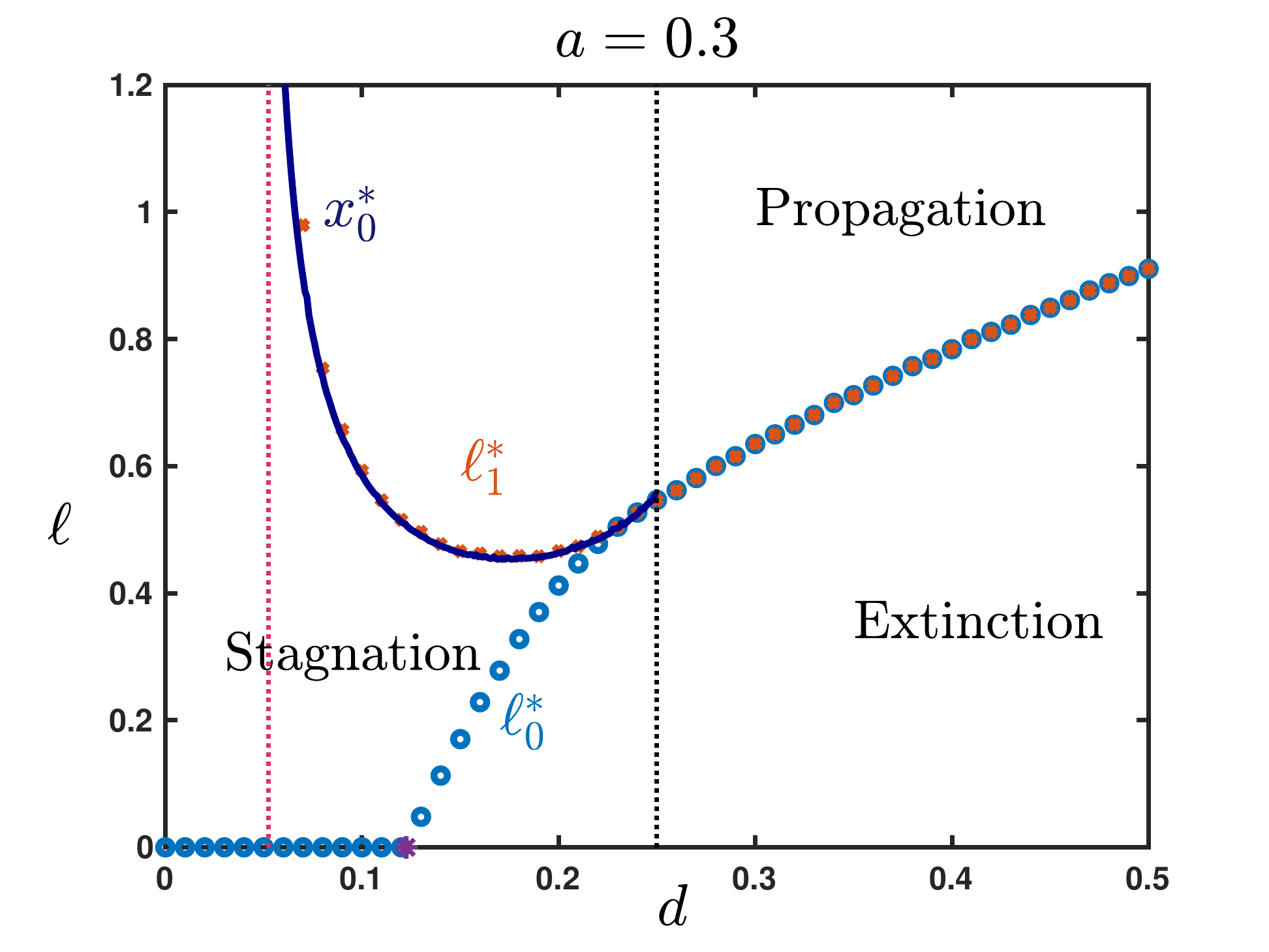}
  \includegraphics[width=.475\textwidth]{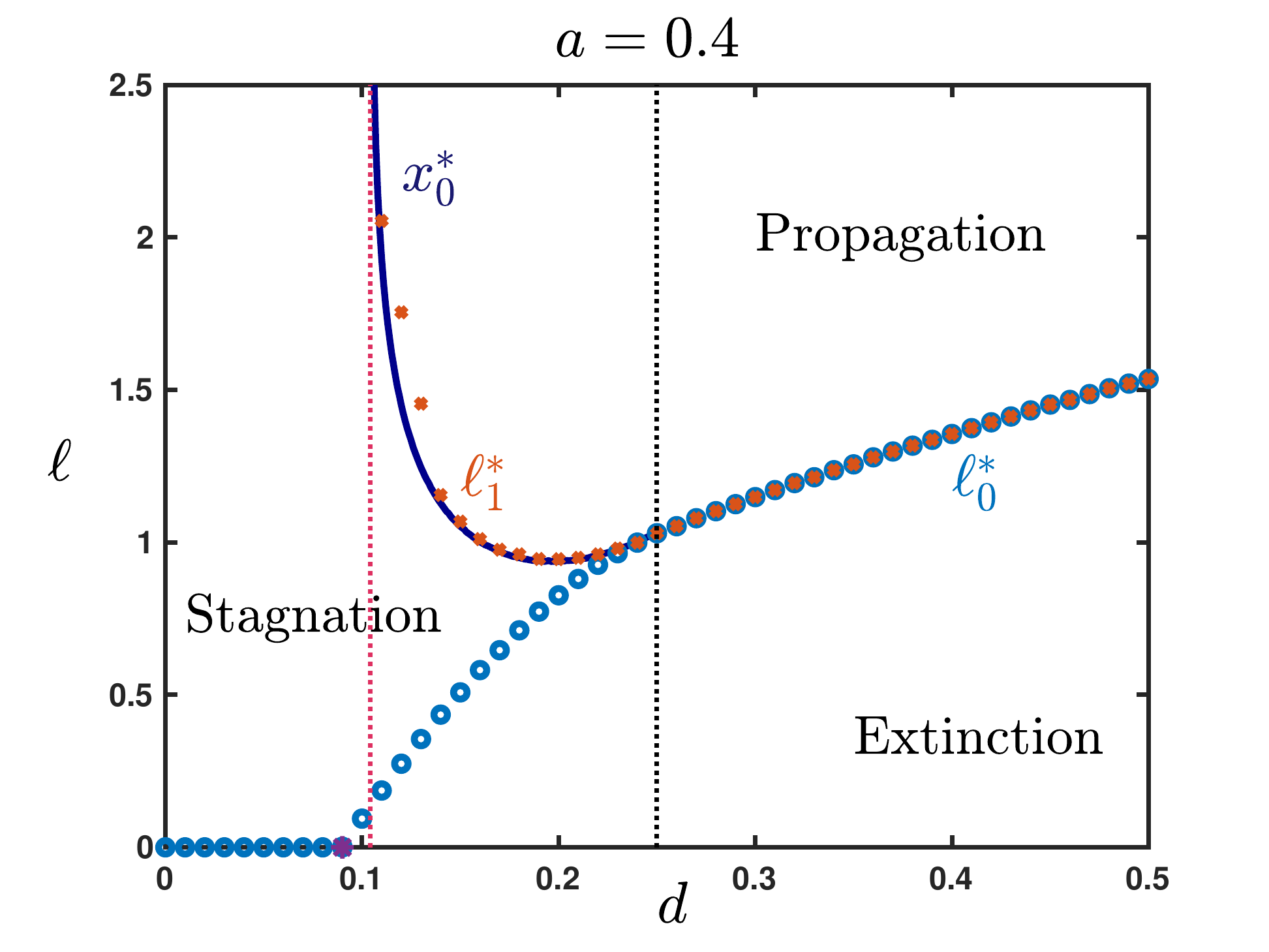}
  \caption{For fixed $a\in(0,1/2)$, we report the long time behavior of the solutions $u_\ell$ of the Cauchy problem for each $d\in(0,1/2)$ as $\ell>0$ is varied. The long time behavior is either extinction, propagation or stagnation. We obtain sharp thresholds between each regime. In the legend of each figure, $\ell_0^*$ represents the onset for extinction while $\ell_1^*$ is the onset for propagation. In all cases, we have that whenever $d\geq 1/2$, there is a sharp threshold between extinction and propagation with $\ell_0^*=\ell_1^*>0$. The pink doted line represents at each $a$ the boundary of the pinning region, below which no propagation can occur. We also remark that extinction only occurs above some threshold value which is marked by a magenta star.}
  \label{fig:seuil}
\end{figure}

We report in this section our results obtained by direct numerical simulations of the Cauchy problem \eqref{edp} using the numerical algorithm depicted in the previous section. Throughout, we have set the nonlinearity to be the cubic $f_a(u)=u(1-u)(u-a)$ with $a\in(0,1/2)$ as parameter, and the kernel to be the exponential function $\K(x)=\mathrm{e}^{-|x|}/2$. For each fixed $a\in(0,1/2)$ and $d \in(0,1/2)$, we have solved \eqref{edp} for each $\ell \in (0,\ell_{max})$ using different criteria to stop the computations to discriminate between the different possible asymptotic behaviors: propagation, extinction or stagnation.

\paragraph{Numerical strategy to evaluate $\ell_0^*$ and $\ell_1^*$.} To numerically compute $\ell_0^*$, whose definition is given by 
\bqs
\ell_0^*= \sup\left\{ \ell>0 ~|~  \underset{t\to+\infty}{\lim} u_\ell(t,\cdot)=0 \text{ uniformly in } \R \right\},
\eqs
for each fixed $a\in(0,1/2)$ and $d \in(0,1/2)$, we ran our numerical scheme from $t=0$ to some final time $t=T_f$ for all $\ell \in (0,\ell_{max})$  . For each $\ell>0$, we evaluated the maximum of the numerically computed solution at the last time step. We then discriminated the precise value of $\ell$ for which there is a transition from extinction, maximum being close to zero, to either propagation or stagnation, maximum being strictly larger than the unstable steady state $a$. Note that if the solution were to pass below $a$ before the final time $T_f$, we stop the computation and declare that there is extinction in that case. In all our figures, $\ell_0^*$ is represented by blue circle. We adopted a similar strategy for the computation of $\ell_1^*$, whose definition is given by 
\bqs
\ell_1^*= \inf\left\{ \ell>0 ~|~  \underset{t\to+\infty}{\lim} u_\ell(t,\cdot)=1 \text{ locally uniformly in } \R \right\}.
\eqs
To discriminate more accurately between stagnation and propagation, we measured the mass of the solution and declared propagation when we simultaneously had an increasing mass and a corresponding solution with maximum being close to one at the final time on a fixed interval. In all our figures, $\ell_1^*$ is represented by red stars.

\begin{figure}[t!]
  \centering
  \includegraphics[width=.475\textwidth]{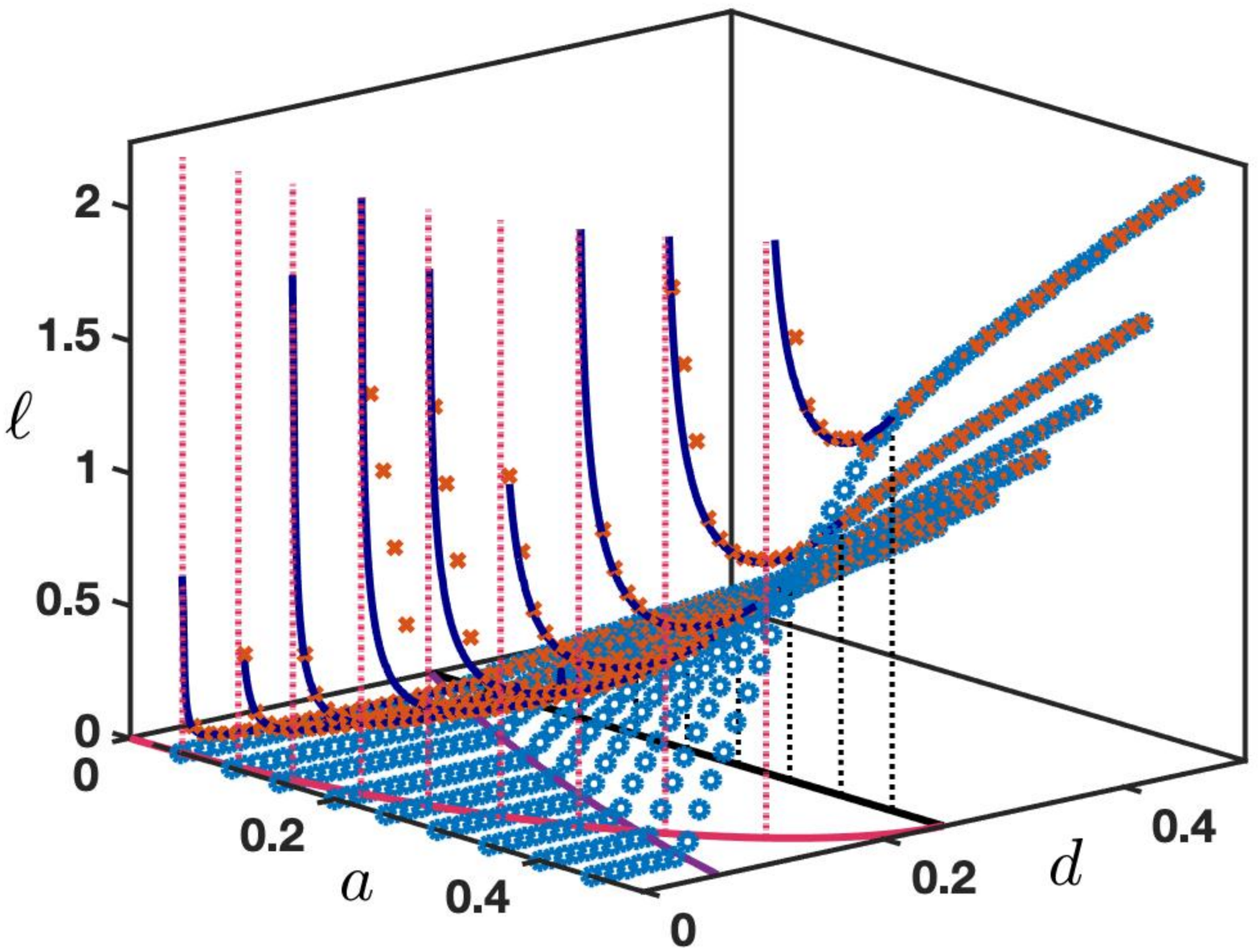}
  \includegraphics[width=.475\textwidth]{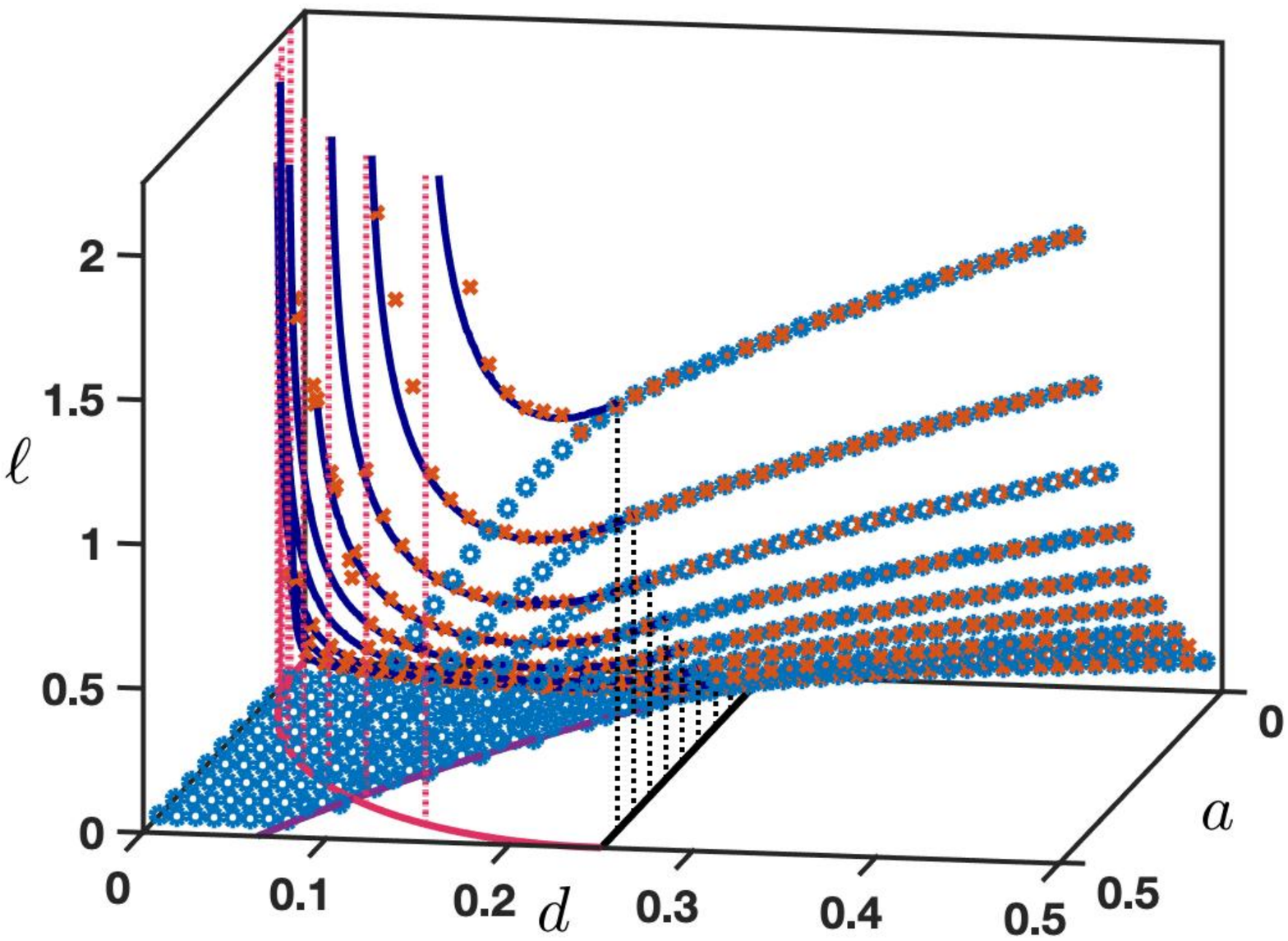}
  \caption{A three dimension visualization of the results presented in Figure~\ref{fig:seuil} with two different views.}
  \label{fig:seuil3D}
\end{figure}

\paragraph{Sharp thresholds.} In Figure~\ref{fig:seuil}, we show typical diagrams for values of the parameter $a$ in $\left\{0.1,0.2,0.3,0.4\right\}$, while in Figure~\ref{fig:seuil3D} we present a three dimension visualization of the results. In each panel of Figure~\ref{fig:seuil}, we have represented by a doted pink vertical line the boundary of the the pinning region, while the black doted vertical line is always at $d=1/4$. We first observe that $\ell_0^*=0$ whenever $d<d_\mathrm{ext}(a)$ where $d_\mathrm{ext}(a)=\frac{(1-a)^2}{4}$ is represented by the purple star in the figures, while $0<\ell_0^*<+\infty$ for any $d>d_\mathrm{ext}(a)$. This numerically indicates that the condition of Proposition~\ref{propext} is sharp. On the other hand, a simple computation shows that if $\gamma\in(a,1)$ is the minimum of $g$, that is $g'(\gamma)=0$, we get that $g(\gamma)=0$ if and only if $d=d_\mathrm{ext}(a)$. As a consequence, the condition $g(u_+)<0$ derived in Proposition~\ref{prop51} is not sharp, since our numerics indicate that $g(\gamma)<0$ (or equivalently $d<d_\mathrm{ext}(a)$) is the sharp condition. Next we observe that  for small values of $0<d<d_\mathrm{pin}(a)$ we have $\ell_1^*=+\infty$, while for $d>d_\mathrm{pin}(a)$ one gets $0<\ell_1^*<+\infty$, with $d_{\mathrm{pin}}(a)=\frac{1}{3}\left( 1-a+a^2-\sqrt{1-2a}\right)$ for $a\in(0,1/2)$. As a consequence, we numerically corroborate the condition of Proposition~\ref{proppropa} which assesses that in order to have propagation one needs to be in the region of parameter space where there exists a traveling front solution with non zero wave speed: this is precisely characterized by the condition $d>d_\mathrm{pin}(a)$ of being outside the pinning region. We further note that, independently of the value of $a\in(0,1/2)$, and as long as $d>1/4$, there is a sharp threshold between propagation and extinction, that is $0<\ell^*:=\ell_0^*=\ell_1^*<+\infty$, with extinction for $\ell \in(0,\ell^*)$ and propagation for $\ell>\ell^*$. We denote this region of parameter space as $\mathcal{R}_1:=\left\{ (a,d) ~|~d>1/4 \text{ and } 0<a<1/2 \right\}$. We remark that this is the usual result that one obtains in the local diffusion case \cite{Zlatos,DM10,MZ13}. Below $d<1/4$ several different situations can happen. 

\begin{figure}[t!]
  \centering
  \includegraphics[width=.475\textwidth]{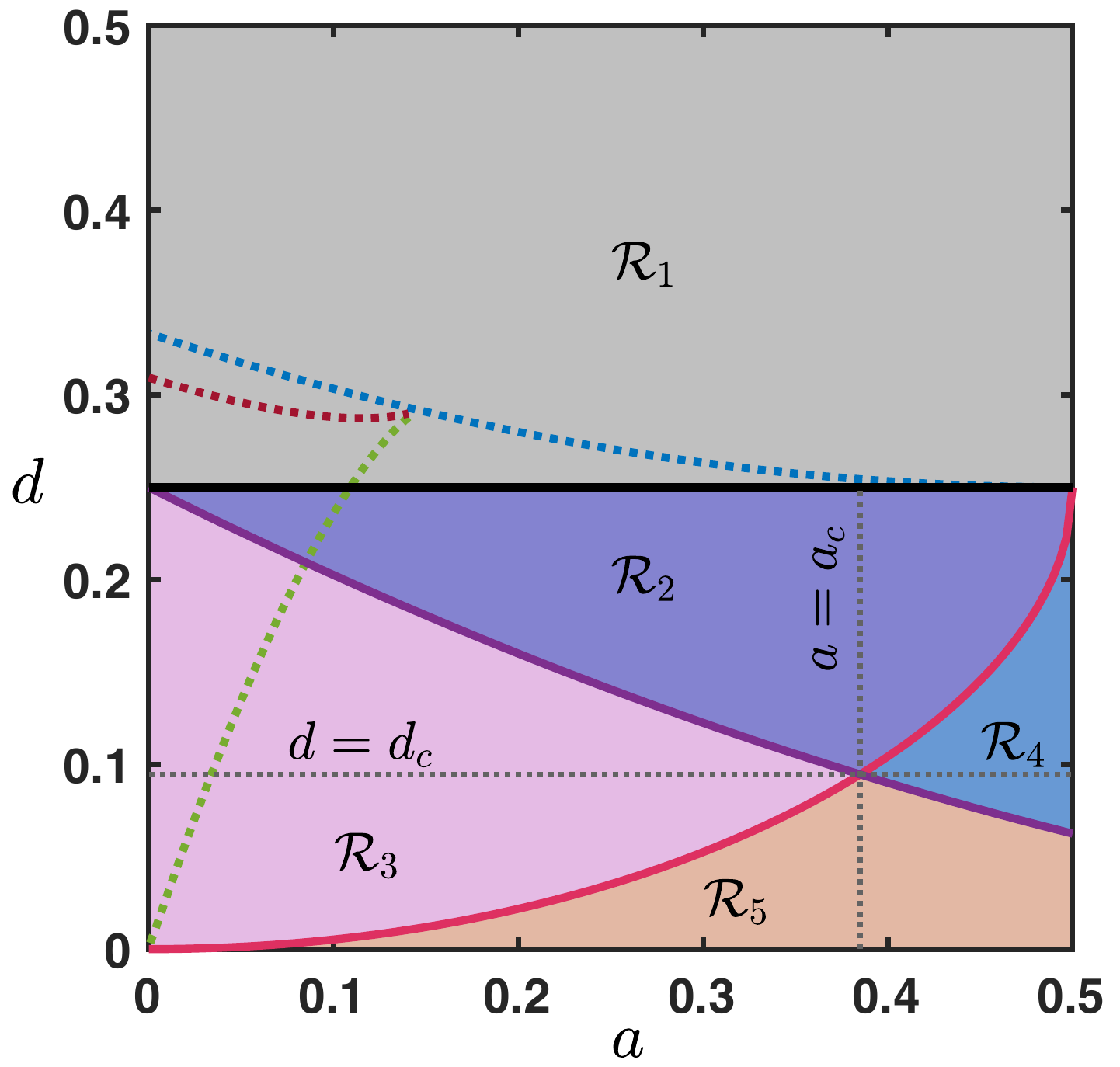}
  \caption{Numerically computed regions $\mathcal{R}_j$, $j=1,\cdots,5$, for the cubic nonlinearity and the exponential kernel $\K(x)=\exp(-|x|)/2$. The boundary between regions $\mathcal{R}_1$ and $\mathcal{R}_2$ appear to be given by $d=1/4$ (black curve). The boundary between $\mathcal{R}_2\cup\mathcal{R}_3$ and $\mathcal{R}_4\cup\mathcal{R}_5$ coincides with the boundary of the pinning region (pink curve) given by $d=d_{\textrm{pin}}(a)$. Finally, the boundary between $\mathcal{R}_2\cup\mathcal{R}_4$ and $\mathcal{R}_3\cup\mathcal{R}_5$ coincides with the parametrized curve given by $d=d_{\textrm{ext}}(a)$ (magenta curve). Note that both curves intersect at a unique point $(a_c,d_c)\simeq (0.3850,0.9455)$. The three dashed lines correspond to the three curves of Figure~\ref{fig:DiagGS} associated to the existence of ground state solutions.}
  \label{fig:RegionsRj}
\end{figure}

Let us introduce $a_c\in(0,1/2)$ solution of $d_{\mathrm{pin}}(a)=d_{\mathrm{ext}}(a)$, for which we have the approximation $a_c\simeq 0.3850$. For $0<a<a_c$, we have that $0<d_{\mathrm{pin}}(a)<d_{\mathrm{ext}}(a)<1/4$ and:
\begin{itemize}
\item for $d_{\mathrm{ext}}(a)<d<1/4$, we have this time $0<\ell_0^*<\ell_1^*<\infty$, such that for all $\ell \in(0,\ell_0^*)$, there is extinction, for all $\ell \in(\ell_0^*,\ell_1^*)$ there is stagnation, and for all $\ell>\ell_1^*$ there is propagation. We denote this region of parameter space as $\mathcal{R}_2^{<}:=\left\{d_{\mathrm{ext}}(a) <d<1/4 \text{ and } 0<a<a_c \right\}$.
\item for $d_{\mathrm{pin}}(a)<d<d_{\mathrm{ext}}(a)$, we have $\ell_0^*=0$ with $0<\ell_1^*<\infty$, such that for all $\ell \in(0,\ell_1^*)$ there is stagnation, and for all $\ell>\ell_1^*$, there is propagation. We denote this region of parameter space $\mathcal{R}_3:=\left\{d_{\mathrm{pin}}(a)< d<d_{\mathrm{ext}}(a)  \text{ and } 0<a<a_c \right\}$.
\item for $0<d<d_{\mathrm{pin}}(a)$, we have $\ell_0^*=0$ with $\ell_1^*=\infty$ such that there is only stagnation with convergence asymptotically towards a ground state solution for each $\ell>0$. We denote this region of parameter space as $\mathcal{R}_5^{<}:=\left\{0<d<d_{\mathrm{pin}}(a) \text{ and } 0<a<a_c \right\}$.
\end{itemize}
On the other hand for $a_c<a<1/2$, we have that $0<d_{\mathrm{ext}}(a)<d_{\mathrm{pin}}(a)<1/4$ and:
\begin{itemize}
\item for $d_{\mathrm{pin}}(a)<d<1/4$, we have this time $0<\ell_0^*<\ell_1^*<\infty$, such that for all $\ell \in(0,\ell_0^*)$, there is extinction, for all $\ell \in(\ell_0^*,\ell_1^*)$ there is stagnation, and for all $\ell>\ell_1^*$ there is propagation. We denote this region $\mathcal{R}_2^{>}:=\left\{d_{\mathrm{pin}}(a)<d<1/4 \text{ and }a_c<a<1/2 \right\}$.
\item for $d_{\mathrm{ext}}(a)<d<d_{\mathrm{pin}}(a)$, we have $0<\ell_0^*<\infty$ with $\ell_1^*=\infty$, such that for all $\ell \in(0,\ell_0^*)$ there is extinction, and for all $\ell>\ell_0^*$, there is stagnation. We denote this region  $\mathcal{R}_4:=\left\{d_{\mathrm{ext}}(a)< d<d_{\mathrm{pin}}(a)  \text{ and } a_c<a<1/2 \right\}$.
\item for $0<d<d_{\mathrm{ext}}(a)$, we have $\ell_0^*=0$ with $\ell_1^*=\infty$ such that there is only stagnation for each $\ell>0$. We denote this region  $\mathcal{R}_5^{>}:=\left\{0<d<d_{\mathrm{ext}}(a) \text{ and } a_c<a<1/2 \right\}$.
\end{itemize}
Finally, at $a=a_c$, we have that $0<d_c:=d_{\mathrm{ext}}(a_c)=d_{\mathrm{pin}}(a_c)<1/4$ and:
\begin{itemize}
\item for $d_c<d<1/4$, we have $0<\ell_0^*<\ell_1^*<\infty$, such that for all $\ell \in(0,\ell_0^*)$, there is extinction, for all $\ell \in(\ell_0^*,\ell_1^*)$ there is stagnation, and for all $\ell>\ell_1^*$ there is propagation. We denote this region $\mathcal{R}_2^{=}:=\left\{d_c<d<1/4 \text{ and } a=a_c \right\}$.
\item for $0<d<d_c$, we have $\ell_0^*=0$ with $\ell_1^*=\infty$ such that there is only stagnation for each $\ell>0$. We denote this region  $\mathcal{R}_5^{=}:=\left\{0<d<d_c \text{ and } a=a_c\right\}$.
\end{itemize}
As a consequence, we can partition the parameter space $(a,d)$ into five different regions, denoted $\mathcal{R}_j$ whose boundaries are given by the horizontal line $d=1/4$ and the two curves $d=d_{\mathrm{pin}}(a)$ and $d=d_{\mathrm{ext}}(a)$. Here we have set $\mathcal{R}_{2,5}:=\mathcal{R}_{2,5}^{<}\cup\mathcal{R}_{2,5}^{=}\cup\mathcal{R}_{2,5}^{>}$. We refer to Figure~\ref{fig:RegionsRj} for an illustration and to Table~\ref{table} for a summary of the possible asymptotic behaviors in each region.

\begin{table}[t!]
\begin{center}
\begin{tabular}{c c c c}
\hline
\hline
Region & \parbox[t]{1.1in}{\centering Propagation} & \parbox[t]{1.3in}{\centering Extinction} & \parbox[t]{1.2in}{\centering Stagnation}  \\
\hline
\hline
$\mathcal{R}_{1}$ & yes & yes & no  \\
 $\mathcal{R}_{2}$ & yes & yes & yes  \\
 $\mathcal{R}_{3}$ & yes & no & yes  \\
  $\mathcal{R}_{4}$ & no & yes & yes   \\
    $\mathcal{R}_{5}$ & no & no & yes   \\
 \hline
 \hline
\end{tabular}
\end{center}
\caption{Overview of different types of possible asymptotic behaviors for the solutions of the Cauchy problem \eqref{edp} for the cubic nonlinearity and the exponential kernel $\K(x)=\exp(-|x|)/2$ for open connected intervals of the parameter $\ell>0$. Regions are outlined in parameter space $(a,d)$ in Figure~\ref{fig:RegionsRj}, below.}
\label{table}
\end{table}

\paragraph{An analytical formula for $\ell_1^*$ in $\mathcal{R}_2\cup\mathcal{R}_3$.} In the region $\mathcal{R}_2\cup\mathcal{R}_3$, that is for $d_\mathrm{pin}(a)<d<1/4$ with $0<a<1/2$ where $0<\ell_1^*<\infty$, we manage to derive an analytical formula for $\ell_1^*$ which is represented by the dark blue curve in Figure~\ref{fig:seuil} and labelled $x_0^*$. Since $0<d<1/4$, we are in the region of parameters where $0<\beta<\gamma<1$ with $g'(\beta)=g'(\gamma)=0$. Note also that from Subsection~\ref{subsecapplcubic} and Figure~\ref{fig:DiagGS}, the region $\mathcal{R}_2\cup\mathcal{R}_3$ is strictly contained within the region of parameters where $\mathcal{G}_-(g(\gamma))<0$. As a consequence, using Lemma~\ref{lem43} and Lemma~\ref{lem45}, for all each $(a,d) \in \mathcal{R}_2\cup\mathcal{R}_3$ such that $\mathcal{G}_-(g(\beta)) \leq 0$, there exists a parameter family of ground state solutions with points of discontinuity at $\pm x_0(v_0)$ for each $v_0\in[g(\gamma),g(\beta)]$ and $v_0>0$. In that case, we define
\bqq
x_0^*:=\underset{v_0\in[g(\gamma),g(\beta)],~v_0>0}{\sup}~ x_0(v_0).
\label{eqx0s1}
\eqq
On the other hand, using Lemma~\ref{lem44}(ii), for all each $(a,d) \in \mathcal{R}_2\cup\mathcal{R}_3$ such that $\mathcal{G}_-(g(\beta))> 0$, there exists a parameter family of ground state solutions with point discontinuity at $\pm x_0(v_0)$ for each $v_0\in[g(\gamma),v_m]$ and $v_0>0$ where $v_m\in(0,g(\beta))$ is solution of $\mathcal{G}_-(v_m)= 0$. In that case, we define
\bqq
\label{eqx0s2}
x_0^*:=\underset{v_0\in[g(\gamma),v_m],~v_0>0}{\sup}~ x_0(v_0).
\eqq
In both cases, $x_0^*>0$ is the largest right point of discontinuity of the family of discontinuous ground state solutions. In fact, it is possible to extrapolate the value $v_0^*>0$ at which the maximum is attained. We first remark that in region $\mathcal{R}_2\cup\mathcal{R}_3$ where $\mathcal{G}_-(g(\beta)) \leq 0$, we always have that $a\leq g(\beta)$ with equality along the curve $d=a-a^2$. Next, we remark that the condition $g(\gamma)=a$ is equivalent to $d=1/4$, such that for any $(a,d)\in \mathcal{R}_2\cup\mathcal{R}_3$ with $\mathcal{G}_-(g(\beta)) \leq 0$, we have $g(\gamma)< a \leq g(\beta)$. Furthermore, the map $v\mapsto v\mapsto \sqrt{-2\mathcal{G}_-(v)}$ defined on $[0,g(\beta)]$ always has a unique maximum at $v=a$ since $\mathcal{G}_-'(v)=-v+g_-^{-1}(v)$. We find that the maximum is achieved precisely for $v_0>0$ which maximizes $v\mapsto \sqrt{-2\mathcal{G}_-(v)}$, that is $v_0^*=a$ and $x_0^*=x_0(a)>0$. A similar argument applies to the other case by noticing that we always have $a\leq v_m$ since $\mathcal{G}_-(v_m)=0$.

Note that the value $x_0(v_0)$ can be numerically computed by evaluating the integral formula \eqref{defx0}, and we have used a standard trapezoidal rule. In each panel of Figure~\ref{fig:seuil}, we have a remarkable agreement between our numerically computed value of $\ell_1^*$ by direct simulations of the long time behavior of the solutions of the Cauchy problem \eqref{edp} and our conjectured formula through the expression of $x_0^*$. This allows us to formulate the following conjecture. 

\begin{conj} Assume that $\K(x)=\mathrm{e}^{-|x|}/2$ and that $f_a(u)=u(1-u)(u-a)$ with $0<a<1/2$. Let $(a,d)\in \mathcal{R}_2\cup\mathcal{R}_3$ such that $0<\ell_1^*<\infty$ is well defined from Proposition~\ref{proppropa}. Let $x_0^*>0$ be the right point of discontinuity of the discontinuous ground state solution associated to $v_0=a$. Then we have $\ell_1^*=x_0^*$.
\end{conj}

\begin{rmk}
From the discussion above, we remark that the boundary between regions $\mathcal{R}_1$ and $\mathcal{R}_2$ given by $d=1/4$ should instead be reinterpreted with the equivalent condition $g(\gamma)=a$.
\end{rmk}

\subsection{Beyond exponential kernels}

\begin{figure}[t!]
  \centering
  \includegraphics[width=.5\textwidth]{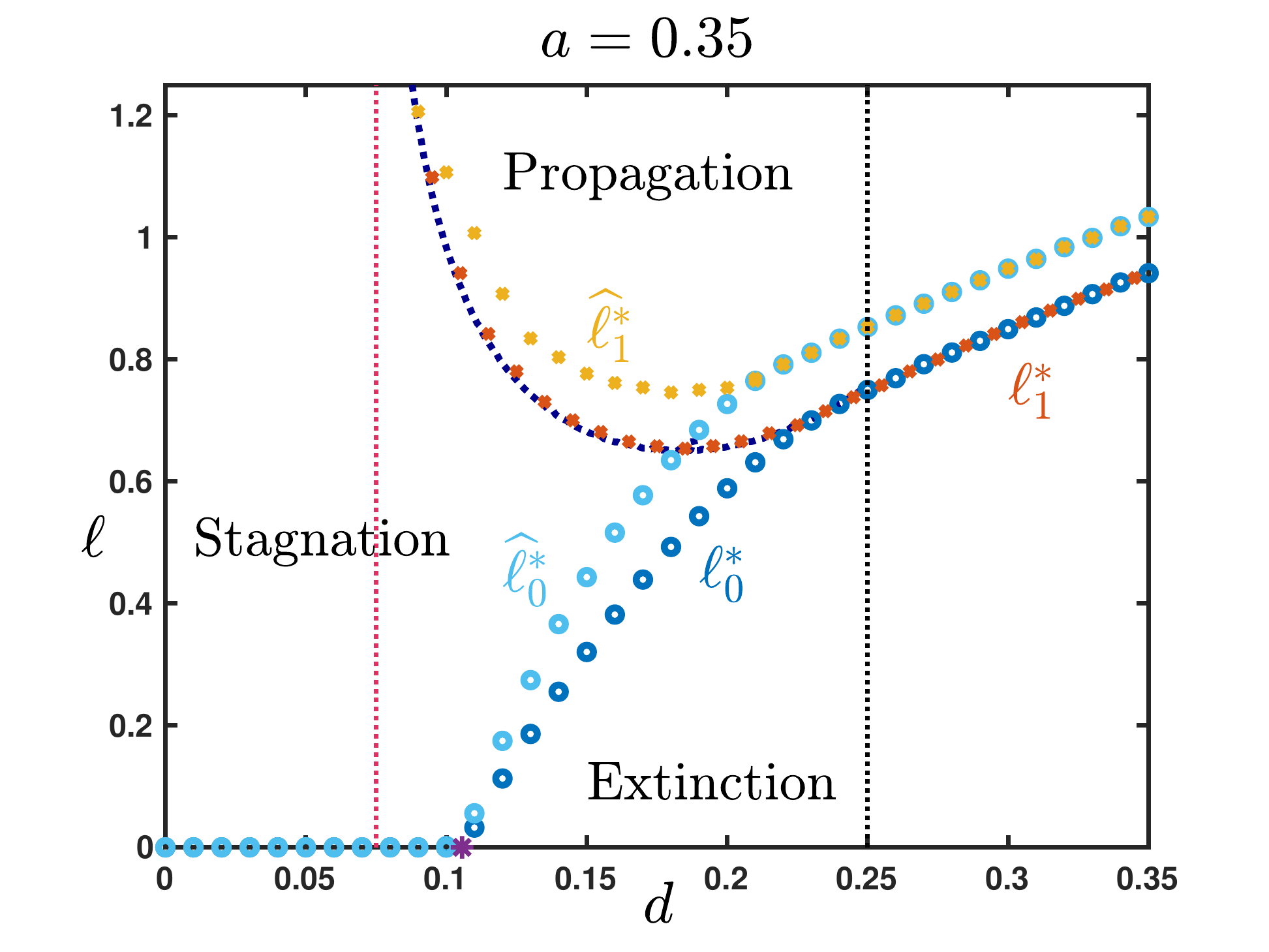}
  \caption{Numerically computed values of $\widehat{\ell}_0^*$ (light blue circles) and  $\widehat{\ell}_1^*$ (yellow stars) for the Gaussian kernel $\K(x)=\mathrm{e}^{-x^2/4}/\sqrt{4\pi}$ for fixed $a=0.35$ and the cubic nonlinearity. We also reported the numerically computed values of $\ell_0^*$ (blue circles) and  $\ell_1^*$ (red stars) for the exponential kernel $\K(x)=\mathrm{e}^{-|x|}/2$ for comparison. Both $\ell_0^*$ and $\widehat{\ell}_0^*$ become positive at the same value for $d$ indicated by the magenta star given by $d_\mathrm{ext}(a)$ with $a=0.35$. On the other hand $\ell_1^*$ and $\widehat{\ell}_1^*$ become bounded only for those values of $d$ larger than $d_\mathrm{pin}(a)$ with $a=0.35$, indicated by the pink dashed vertical line. We obtain that $\widehat{\ell}_0^*=\widehat{\ell}_1^*$ for all $d\geq 1/4$. Note that for the Gaussian kernel the cusp between the two curves $\widehat{\ell}_0^*$ and  $\widehat{\ell}_1^*$ near $d=1/4$ is more degenerate.}
  \label{fig:SeuilGauss}
\end{figure}

We also numerically studied the case of a kernel given by a Gaussian $\K(x)=\mathrm{e}^{-x^2/4}/\sqrt{4\pi}$ with Fourier transform $\widehat{\K}(\xi)=\mathrm{e}^{-\xi^2}$, still with the cubic nonlinearity. We found that the regions $\mathcal{R}_j$, $j=1,\cdots,5$ remain unchanged. More precisely, if we denote by $\widehat{\ell}_0^*$ and  $\widehat{\ell}_1^*$  the thresholds for extinction and propagation respectively in the case of the Gaussian kernel, then
\begin{itemize}
\item in region $\mathcal{R}_1$, we have $0<\widehat{\ell}_0^*=\widehat{\ell}_1^*<\infty$;
\item in region $\mathcal{R}_2$, we have $0<\widehat{\ell}_0^*<\widehat{\ell}_1^*<\infty$;
\item in region $\mathcal{R}_3$, we have $0=\widehat{\ell}_0^*$ and $0<\widehat{\ell}_1^*<\infty$;
\item in region $\mathcal{R}_4$, we have $0<\widehat{\ell}_0^*<+\infty$ and $\widehat{\ell}_1^*=\infty$;
\item in region $\mathcal{R}_5$, we have $0=\widehat{\ell}_0^*$ and $\widehat{\ell}_1^*=\infty$.
\end{itemize}
We refer to Figure~\ref{fig:SeuilGauss} for a typical situation at the fixed value $a=0.35$. In the Gaussian case, determining the exact location in parameter space for which $\widehat{\ell}_0^*=\widehat{\ell}_1^*$ occurs is a much more delicate task than in the exponential case. Indeed, the cusp between the two curves $\widehat{\ell}_0^*$ and  $\widehat{\ell}_1^*$ is more degenerate. Actually,  we conjecture that the regions $\mathcal{R}_j$ are universal among the class of localized kernels with normalized second moment $\frac{1}{2}\int_\R x^2 \K(x)\md x=1$ and when $f$ is given by the cubic nonlinearity.

\begin{conj}\label{conjRj}
Assume that $f_a(u)=u(1-u)(u-a)$ with $0<a<1/2$. Let $\K$ satisfy Hypothesis~\ref{hypf} and Hypothesis~\ref{hypexploc} with normalized second moment $\frac{1}{2}\int_\R x^2 \K(x)\md x=1$. Let $\mathcal{R}_j$, $j=1,\cdots,5$, be the regions defined above. Then we have:
\begin{itemize}
\item[(i)] In region $\mathcal{R}_1$, there is a sharp threshold between extinction and propagation, that is $0<\ell_0^*=\ell_1^*<\infty$.
\item[(ii)] In region $\mathcal{R}_2$, there are two sharp thresholds between extinction and stagnation on the one hand and between  stagnation and propagation on the other hand, that is $0<\ell_0^*<\ell_1^*<\infty$;
\item[(iii)] In region $\mathcal{R}_3$,  there is a sharp threshold between stagnation and propagation, that is $0=\ell_0^*$ and $0<\ell_1^*<\infty$;
\item[(iv)] In region $\mathcal{R}_4$,  there is a sharp threshold between extinction and stagnation, that is $0<\ell_0^*<+\infty$ and $\ell_1^*=\infty$;
\item[(v)] In region $\mathcal{R}_5$, there is always stagnation, that is $0=\ell_0^*$ and $\ell_1^*=\infty$.
\end{itemize}
\end{conj}

Proving such a conjecture is a major open problem that we leave for future work. We note that \cite{ZL22} has proved Conjecture~\ref{conjRj}(i) when the kernel is compactly supported, and in a sub-region of $\mathcal{R}_1$, namely the sub-region above the curve $d=\frac{1-a+a^2}{3}$ which corresponds to the region where $g'>0$ (indicated by the dashed blue lin in Figure~\ref{fig:RegionsRj}).

\section*{Acknowledgements} This research was initiated during Summer 2021 via an {\em intensive research internship program} funded by Labex CIMI under grant agreement ANR-11-LABX-0040. G. Faye also acknowledges support from the ANR via the project Indyana under grant agreement ANR- 21- CE40-0008 and from an ANITI (Artificial and Natural Intelligence Toulouse Institute) Research Chair.

\newpage

\bibliography{plain}

\end{document}